\documentclass[12pt]{article}

\usepackage[utf8]{inputenc}
\usepackage{amsmath,amssymb,epsfig,bbm}
\usepackage{amsthm}
\usepackage{stmaryrd,mathabx}
\usepackage{comment}
\usepackage{color}

\usepackage{sansmathfonts}
\usepackage[T1]{fontenc}

\usepackage{lmodern}
\usepackage{tikz}
\usepackage{adjustbox}
\usepackage{float}

\usepackage{interval}

\usepackage{enumitem}
\usepackage{varwidth}
\usepackage{csquotes}
\usepackage{mathrsfs}

\usepackage{mathtools}

\usepackage[textsize=scriptsize,textwidth=2.5cm,color=pink]{todonotes}

\newtheorem{defi}{Definition}[section]

\newtheorem{theorem}[defi]{Theorem}

\newtheorem{conjecture}[defi]{Conjecture}
\newtheorem{lemma}[defi]{Lemma}

\newtheorem{corollary}[defi]{Corollary}

\newtheorem{remark}[defi]{Remark}

\newcommand{\A}{{\cal A}}
\newcommand{\B}{{\cal B}}
\newcommand{\C}{{\cal C}}

\newcommand{\F}{{\cal F}}
\newcommand{\G}{{\cal G}}

\newcommand{\OOO}{{\cal O}}

\newcommand{\X}{{\cal X}}
\newcommand{\Y}{{\cal Y}}

\newcommand{\maths}[1]{{\mathbb #1}}  

\newcommand{\CC}{\maths{C}}

\newcommand{\HH}{\maths{H}}

\newcommand{\NN}{\maths{N}}

\newcommand{\QQ}{\maths{Q}}
\newcommand{\RR}{\maths{R}}
\newcommand{\SSS}{\maths{S}}

\newcommand{\ZZ}{\maths{Z}}
\def\11{{\mathbbm 1}}

\newcommand{\mmm}{{\mathfrak m}}

\newcommand{\CCC}{{\mathfrak C}}

\newcommand{\weakstar}{\overset{*}\rightharpoonup}

\newcommand{\bs}{\backslash}

\newcommand{\bigO}{\operatorname{O}}

\newcommand{\card}{{\operatorname{card}}}

\newcommand{\covol}{\operatorname{covol}}
\newcommand{\diam}{{\operatorname{diam}}}

\newcommand{\id}{\operatorname{id}}
\newcommand{\pr}{\operatorname{pr}}

\newcommand{\Leb}{\operatorname{Leb}}

\newcommand{\supp}{\operatorname{supp}}
\newcommand{\ssm}{\smallsetminus}

\newcommand{\PSL}{\operatorname{PSL}}
\newcommand{\SL}{\operatorname{SL}}

\newcommand{\sys}{\operatorname{sys}}

\newcommand*{\transp}[2][-3mu]{\ensuremath{\mskip1mu\prescript{\smash{\mathrm t\mkern#1}}{}{\mathstrut#2}}}%

\usepackage[
backend=biber,
style=alphabetic,
sorting=debug,
]{biblatex}

\renewbibmacro{in:}{%
	\ifentrytype{article}
	{}
	{\bibstring{in}%
		\printunit{\intitlepunct}}}

\DeclareFieldFormat{pages}{#1} 

\newcommand\numberthis{\addtocounter{equation}{1}\tag{\theequation}}

\usepackage{orcidlink}

\usepackage{hyperref}
\hypersetup{
	pdfencoding=auto,
	colorlinks=true,
	citecolor=orange,
	filecolor=black,
	linkcolor=blue,
	urlcolor=blue,
	bookmarksnumbered=true}

\begin{document}
\thispagestyle{plain}
\begin{center}
	\Large
	\textbf{Gaps in the complex Farey sequence \\ of an imaginary quadratic number field}
	
	\large
	\vspace{0.3cm}
	\hspace{0.35cm} Rafael Sayous \orcidlink{0009-0007-6306-8546}
	
	\normalsize
	\vspace{0.1cm}
	\today
\end{center}

\begin{abstract}
	Given an imaginary quadratic number field $K$ with ring of integers $\OOO_K$, we are interested in the asymptotic \emph{distance to nearest neighbour} (or \emph{gap}) statistic of complex Farey fractions $\frac{p}{q}$, with $p,q \in \OOO_K$ and $0<|q|\leq T$, as $T \to \infty$. Reformulating this problem in a homogeneous dynamical setting, we follow the approach of J.~Marklof for real Farey fractions with several variables \cite{marklof2013finescalestatsfarey} and adapt a joint equidistribution result in the real $3$-dimensional hyperbolic space of J.~Parkkonen and F.~Paulin \cite{parkkonenpaulin2023jointfarey} to derive the existence of a probability measure describing this asymptotic gap statistic. We obtain an integral formula for the associated cumulative distribution function, and use geometric arguments to find an explicit estimate for its tail distribution in the cases of Gaussian and Eisenstein fractions.
\end{abstract}

\section{Introduction}
In order to obtain refined statistical properties of deterministic real sequences $\mod 1$, their asymptotic gap distributions have been extensively studied, see \cite{sos1958threegap,hall1970fareyfraction,elkiesmcmullen2004gapssqrtn,vijay2008elevengap,marklof2013gapslogs,boca2014gapsfareydivisibilityconstraints,athreya2016gaps,rudnickzhang2017gapcirclepacking,bourgain2017spacialstatsphere,haynesmarklof2020higherdim3gaps, lutsko2022fareythinggroups, aistleitner2023gapsmod1}. This distribution remains a mystery in many explicit cases. A surprising known example is given by the sequence $(\sqrt{n} \mod 1)_{n\in\NN}$: it has been shown to exhibit a nonstandard asymptotic gap distribution by N.D.~Elkies and C.T.~McMullen \cite{elkiesmcmullen2004gapssqrtn} (the "standard" one is given by an exponential law). For the other sequences $(n^\alpha \mod 1)_{n\in\NN}$ (with $\alpha \in ]0,1[\,\ssm \{\frac{1}{2}\}$), computing the gap statistic is an open question (and numerically seems to be an exponential law, which is the case for gaps in a sequence of uniform random variables on $[0,1[\,$).

In the present paper, we begin the asymptotic study of gaps for the complex Farey sequence given by fractions with increasing bound of their height in any quadratic imaginary field. Hence we work with gaps in a $2$-dimensional torus and use the distance to nearest neighbour to define them. This has already been a fruitful choice for a generalisation of the Three Gaps Theorem \cite{vijay2008elevengap} and for studying gaps in rescaled vector points on the sphere $\SSS_2$ associated to "sum of three squares" decompositions of large integers \cite[§~1.3]{bourgain2017spacialstatsphere}.

For clarity, in this introduction we focus on the imaginary quadratic field $\QQ(i)$ (see Section \ref{sec:density_gap} for a more general setting). Let $\pr : \CC \to \CC/\ZZ[i]$ denote the canonical projection on the square torus. The \emph{Gaussian Farey sequence with height at most} $T>0$ is defined as the set
$$
\G_T = \Big\{ \pr \Big(\frac{p}{q}\Big) \, : \, p,q \in \ZZ[i], \, 0 < |q| \leq T \Big\} \subset \CC / \ZZ[i].
$$
The sets $\G_{10}$ and $\G_{20}$ can be seen in Figure \ref{fig:gauss_frac}. The Diophantine approximation properties of these fractions have been famously studied  \cite{cassels1951fareyfrac,schmidt1967fareytriangles}. For any point $x$ in $\CC$ or in $\CC/\ZZ[i]$, let $\Delta_x$ denote the Dirac mass at $x$. Denoting by $dx$ the Haar probability measure on the square torus $\CC/\ZZ[i]$, we have the equidistribution theorem (see for instance \cite[Th.~4]{cosentino1999equidibparabpts}), as $T \to +\infty$,
$$
\frac{1}{\card \, \G_T} \sum_{r \in \G_T} \Delta_r \weakstar dx.
$$
\begin{figure}[ht]
	\centering
	\scalebox{1.}{
		\begin{adjustbox}{clip, trim=5cm 0.85cm 4.55cm 1.cm, max width=\textwidth}
			\includegraphics{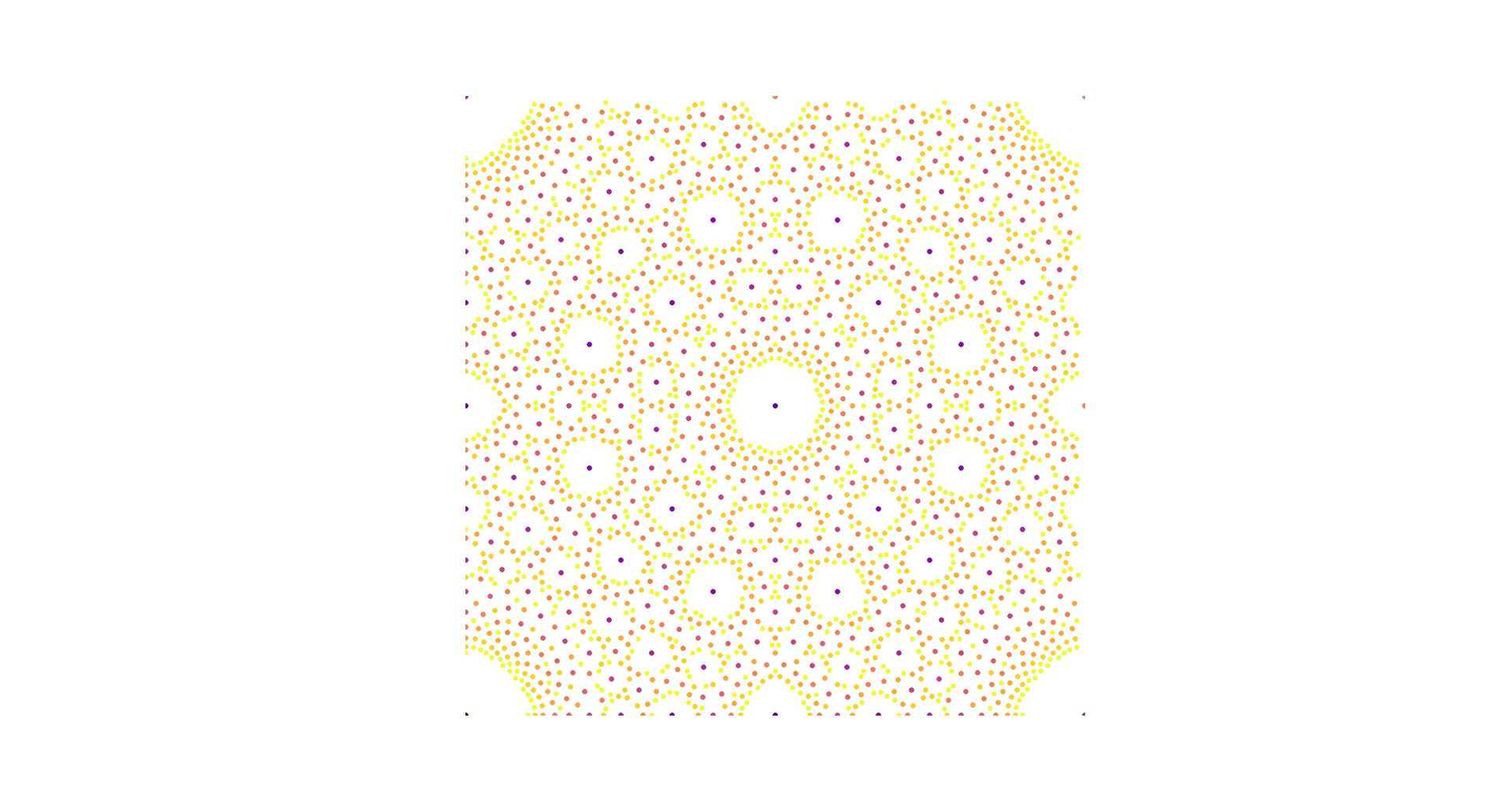}
		\end{adjustbox}
		\hspace{5mm}
		\begin{adjustbox}{clip, trim=5cm 0.85cm 4.55cm 1.cm, max width=\textwidth}
			\includegraphics{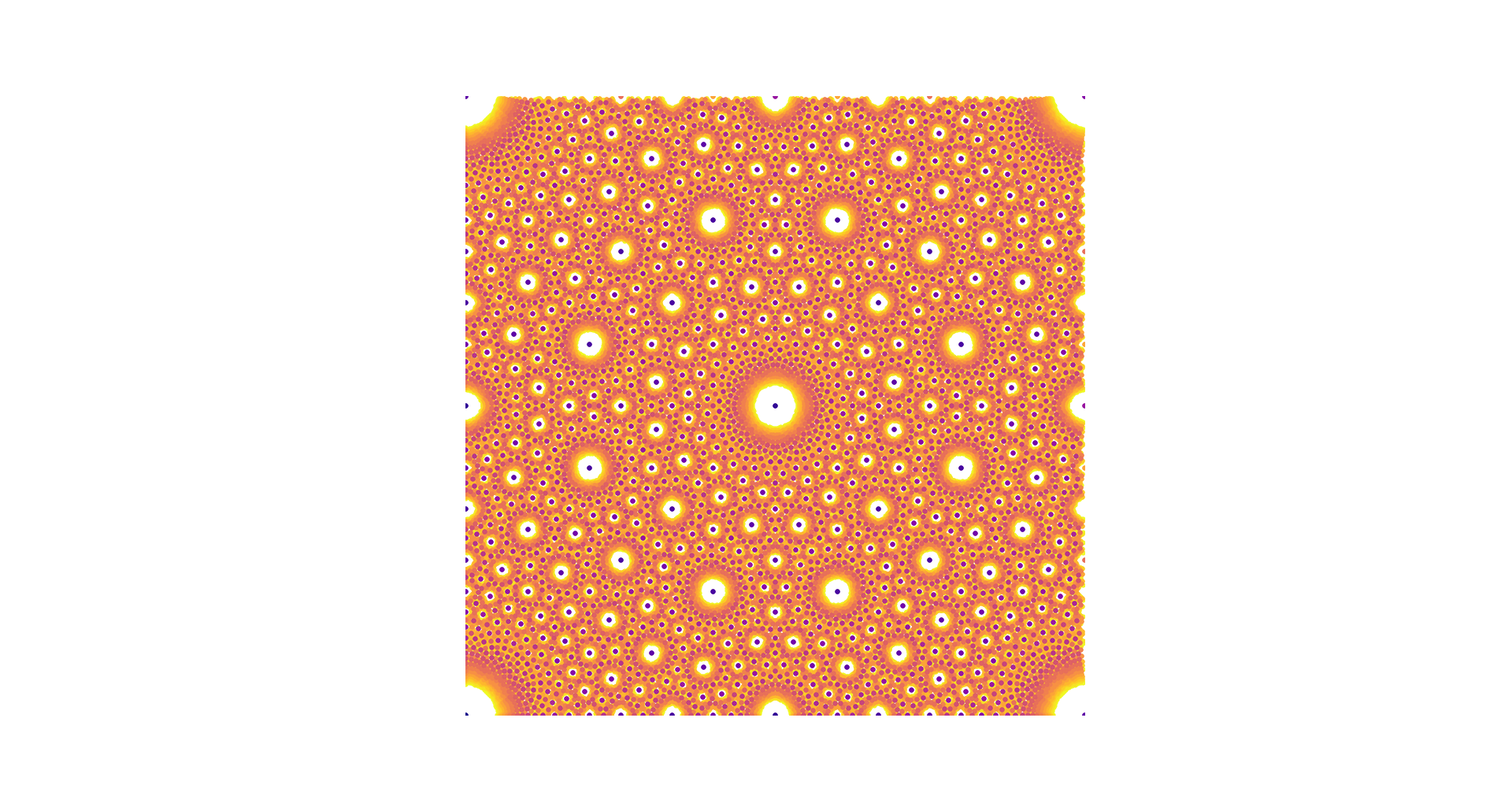}
		\end{adjustbox}
	}
	\caption{Gaussian fractions with height at most $10$ (on the left) and $20$ (on the right) in the square torus $\CC/\ZZ[i]$, and colours of points depending on their height.}
	\label{fig:gauss_frac}
\end{figure}

Let $d$ denote the quotient distance on $\CC/\ZZ[i]$. The \emph{gap statistic} of $(\G_T)_T$ is the asymptotic study of the probability measures
$$
\mu_T = \frac{1}{\card \, \G_T} \sum_{r \in \G_T} \Delta_{T^2 d(r, \G_T\ssm\{r\})}.
$$ 
The terminology "gap" is chosen here instead of "distance to the nearest neighbour" in order to emphasize the similarities between our main result below and the gap (or 'distance to the nearest neighbour to the right") statistic in the classical real case. In the latter formula, the scaling factor $T^2$ is chosen, a priori, in order to obtain a macroscopic rescaled average gap. Indeed, at each time $T \geq 1$, the gaps $\delta_r = d(r, \G_T \ssm \{r\}) >0$ give rise to a family $({\rm D}(r,\frac{\delta_r}{2}))_{r \in \G_t}$ of pairwise disjoint open disks in the square torus $\CC/\ZZ[i]$ (some of which are tangent to each other). We expect the average gap to be comparable to the one we compute in the case of if this family of disks covers the whole torus, namely to be comparable to $2 \sqrt{\frac{1}{\pi \card \, \G_T}} \sim \frac{2 \sqrt{2}}{\pi} \sqrt{\zeta_{\QQ(i)}(2)} \frac{1}{T^2}$ (since $\card \, \G_T \sim \frac{\pi}{2 \zeta_{\QQ(i)}(2)} T^4$ by a generalisation of Mertens' formula \cite[Satz~2]{grotz1979mertenslike}) with $\zeta_{\QQ(i)}$ denoting the Dedekind zeta function of $\QQ(i)$. In other words, we expect the disks ${\rm D}(r,\frac{\delta_r}{2})$ to be close to the Voronoï cells in the tessellation of $\G_T$. And, a posteriori, it is an arguably good choice since there is no loss of mass at infinity in the convergence stated in Theorem \ref{th:intro} below. We notice the inclusion $\supp{ \mu_T} \subset [1,+\infty[$ (since the systole of $\ZZ[i]$ is equal to $1$).

The following result is a description of the asymptotic cumulative distribution of gaps for Gaussian fractions. A version for any imaginary quadratic field is given in our main results Theorem \ref{th:existence_density} and Corollary \ref{cor:tail}.
\begin{theorem}\label{th:intro}
	There exists a probability measure $\mu$ on $[0,+\infty[\,$, absolutely continuous with respect to the Lebesgue measure, such that we have the vague convergence of measures, as $T \to +\infty$,
	$$
	\mu_T \weakstar \mu \, .
	$$
	Furthermore, we have:
	\begin{itemize}
		\item[$\bullet$] the equality of support $\supp{\mu} = [1,+\infty[\,$ ;
		\item[$\bullet$] a description of the cumulative distribution of $\mu$: with $A(z,r,R) \subset \CC$ denoting the closed annulus centred at $z \in \CC$ and of interior and exterior radii $r >0$ and $R>0$ (empty if $R<r$), we have the formula, for every $\delta >0$,
		$$
		\mu([0,\delta]) = 2 \int_{s=0}^{+\infty} dx\Big( \pr \Big( \bigcup_{\substack{p,q \in \ZZ[i] \\ p \neq 0}} A\Big( \frac{q}{p}, \frac{e^s}{\delta}, \frac{e^\frac{s}{2}}{|p|} \Big) \Big) \Big) e^{-2s} \, ds \, ;
		$$
		\item[$\bullet$] a tail estimate for $\mu$ given by, as $\delta \to +\infty$,
		$$
		\mu(\,]\delta,+\infty[\,) = \frac{1}{\delta^4} + \bigO\big(\frac{1}{\delta^5}\big).
		$$
	\end{itemize}
\end{theorem}

\begin{figure}[ht]
	\centering
	\scalebox{0.98}{
		\begin{adjustbox}{clip, trim=2cm 0.85cm 2cm 1.35cm, max width=\textwidth}
			\includegraphics{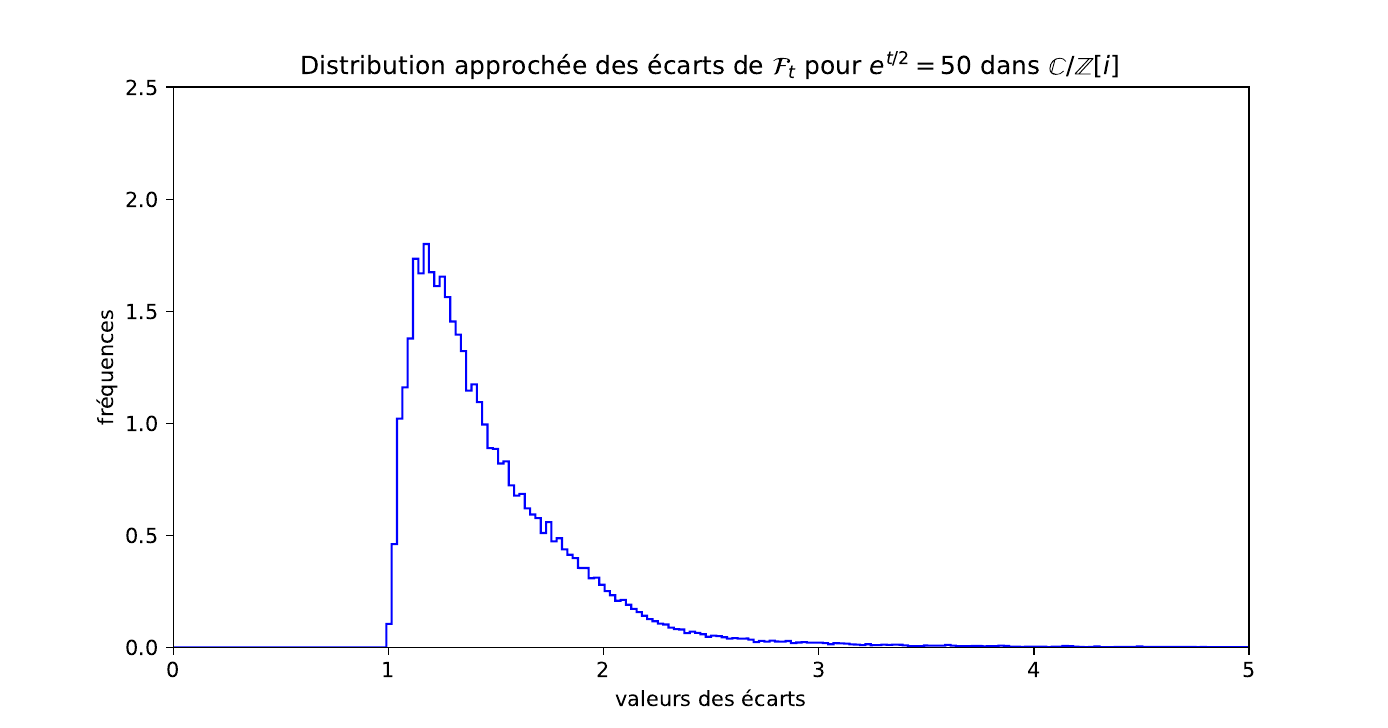}
		\end{adjustbox}
	}
	\caption{A numerical approximation of the asymptotic gap density for Gaussian fractions, using points with height at most $50$. For the approached tail distribution function, see the top graph in Figure \ref{fig:tail_approx_zi}.}
	\label{fig:gauss_frac_density}
\end{figure}

The density of $\mu$, empirically approached in Figure \ref{fig:gauss_frac_density}, is similar to the one found by R.~R.~Hall for gaps in real Farey fractions \cite{hall1970fareyfraction} (see \cite[Fig.~3]{athreya2016gaps}). Gap densities seemingly close to the Hall density have already been found in other settings, e.g.~studying the gaps between tangency points along a line in a circle packing \cite[Fig.~3]{rudnickzhang2017gapcirclepacking} (where the Hall distribution is found by studying gaps for the classical Apollonian packing).

The proof we present of Theorem \ref{th:intro} (and more generally Theorem \ref{th:existence_density} and Corollary \ref{cor:tail}) will follow the approach of J.~Marklof using fine-scale tail functions \cite{marklof2013finescalestatsfarey} and will rely on a lifted version of the joint equidistribution result of J.~Parkkonen and F.~Paulin for complex Farey fractions placed on a horosphere in the $3$-dimensional hyperbolic space $\HH^3_\RR$ when pushed by the geodesic flow \cite[Cor.~4.2]{parkkonenpaulin2023jointfarey}.

We conclude this introduction by mentioning that the result of Theorem \ref{th:intro} still holds, with same limit $\mu$ if we only take into account Gaussian fractions $r \in \G_T$ falling in a fixed Borel set $\B$ which has nonzero measure and neglibible boundary, as we could expect from the self-similarity observed in Figure \ref{fig:gauss_frac}. This fact is given in Theorem \ref{th:existence_density}.

\medskip
\noindent{\small {\it Acknowledgments:} The author would like to thank J.~Parkkonen and F.~Paulin, the supervisors of his ongoing doctorate, for their support through discussion, suggestions and corrections during this research.}

\section{Existence of a density for the gap statistics}\label{sec:density_gap}
Let $K$ be an imaginary quadratic number field in $\CC$, with ring of integers $\OOO_K$, discriminant $D_K$ and Dedekind zeta function $\zeta_K$. Let $\pr_{\OOO_K} : \CC \to \CC/\OOO_K$ denote the canonical projection on the torus. The \emph{complex Farey fractions of height at most} $e^{\frac{t}{2}}$ are defined as the elements of the set
$$
\F_t = \Big\{ \pr_{\OOO_K}\Big(\frac{p}{q}\Big) \, : \, p,q \in \OOO_K, \, p\OOO_K + q\OOO_K = \OOO_K, \, 0 < |q| \leq e^{\frac{t}{2}} \Big\} \subset \CC / \OOO_K.
$$

\begin{figure}[ht]
\centering
\scalebox{1.1}{
	\begin{adjustbox}{clip, trim=2.7cm 1.2cm 2.4cm 1.3cm, max width=\textwidth}
		\includegraphics{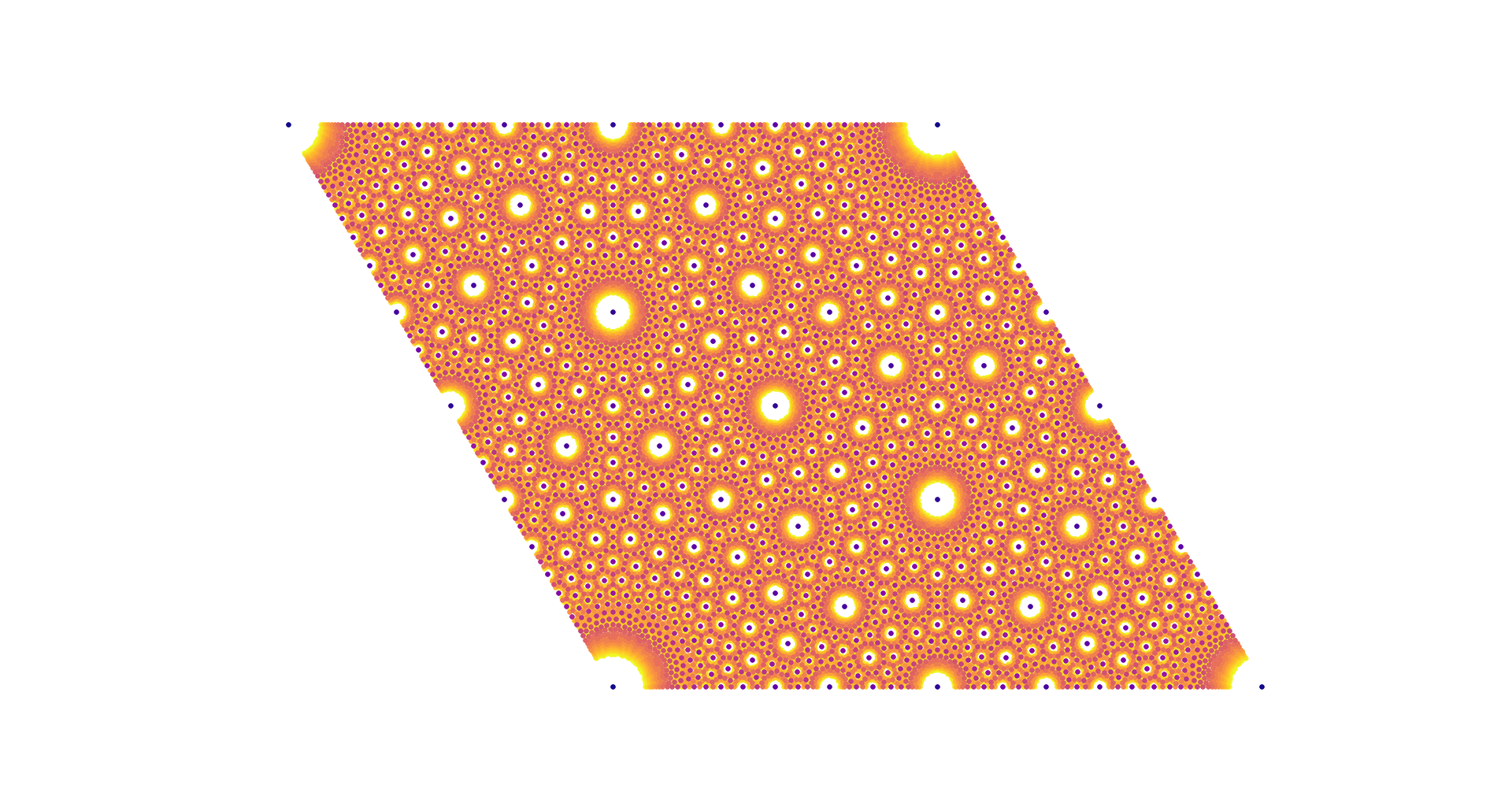}
	\end{adjustbox}
}
\caption{Points of $\F_t$ with $K=\QQ(i\sqrt{3})$ (\emph{Eisenstein fractions}) hence $\OOO_K=\ZZ[e^{i\frac{2\pi}{3}}]$, with height at most $20$ and colours of points depending on their height.}
\label{fig:eis_frac}
\end{figure}

The family $(\F_t)_{t\geq 0}$ equidistributes towards the Haar measure $dx$ on $\CC/\OOO_K$ (of total mass $\covol_{\OOO_K} = \frac{\sqrt{|D_K|}}{2}$) in the following sense: as $t \to \infty$, we have the vague convergence of Borel probability measures on $\CC/\OOO_K$ (see \cite[Satz~2]{grotz1979mertenslike})
$$
\frac{1}{\card \, \F_t} \sum_{r \in \F_t} \Delta_{r} \weakstar \frac{2}{\sqrt{|D_K|}} \; dx.
$$
As $t \to \infty$, Mertens formula yields the estimation (see for instance \cite[Th.~4]{cosentino1999equidibparabpts})
$$
\card \, \F_t \sim c_K e^{2t} \mbox{ where } c_K = \frac{\pi}{\sqrt{|D_K|} \, \zeta_K(2) }.
$$
We are interested in the statistics of the gaps in $(\F_t)_{t \geq 0}$ defined as the distance to the nearest neighbour. More precisely, denoting by $d$ the quotient distance on $\CC/\OOO_K$ obtained from the Euclidean distance on $\CC$, we study the \emph{empirical gap measures of $(\F_t)_{t\geq 0}$} defined as the family of probability measures with general term
{\large
$$
\mu_t = \frac{1}{\card \, \F_t} \sum_{r \in \F_t} \Delta_{e^t d(r, \, \F_t \ssm \{r\})}.
$$
}Therefore, for each real number $t \geq 0$, the measure $\mu_t$ counts "gaps" in the following way: a \emph{gap} around each point $r \in \F_t$ is defined as the maximal radius $\delta_r >0$ for which the open disk ${\rm D}(r,\delta_r) \subset \CC/\OOO_K$ only meets $\F_t$ at its centre $r$. The factor $e^t$ is an adapted scaling parameter. As explained in the introduction, assuming these gaps are somehow close the ones we would obtain using Voronoï tessellations, we expect the average gap to be comparable to $2 \sqrt{\frac{\covol_{\OOO_K}}{\card \, \F_t}} \sim \big(2\frac{\sqrt{|D_K|}}{\pi \, c_K}\big)^{\frac{1}{2}} e^{-t}$.

Inspired by the work of J.~Marklof on the case of the classical Farey fractions in \cite{marklof2013finescalestatsfarey}, our focus lies on the tail distribution function for $(\mu_t)_{t\geq 0}$, defined as follows.

\begin{defi}{\rm
	The \emph{fine-scale tail distribution function for the gaps of $(\F_t)_{t \geq 0}$} is defined, for every integer $k \in \NN$, for all Borel subsets $\A \subset \CC$ and $\B \subset \CC/\OOO_K$ such that $\B$ is nonnegligible and has negligible boundary (i.e.~$dx(\B) >0$ and $dx(\partial \B)=0$), and for every real number $t \geq 0$ large enough so that $\F_t \cap \B \neq \emptyset$, by the formula
	$$
	F_t(k,\A,\B) = \frac{\card \big\{ r \in \F_t \cap \B \, : \, \card{(\F_t \cap (r+\pr_{\OOO_K}(e^{-t}\A)))}=k \big\}}{\card(\F_t \cap \B)}.
	$$
}
\end{defi}
We use the parameter $\B$ in order to focus only on a portion of all the complex Farey fractions. For every time $t\geq 0$ and every Farey fraction $r\in\F_t$, notice that the rescaled gap $e^{t}d(r,\F_t\ssm\{r\})$ is bigger than $\delta$ if, and only if, the set $\F_t$ meets the closed ball $r+\pr_{\OOO_K}(\bar{D}(0,e^{-t} \delta))$ only at its center $r$. Hence, the function $F_t$ describes the tail distribution of $\mu_t$ with the formula
$$
\mu_t(\, ]\delta, +\infty[ \,) = F_t(1,\overline{\rm D}(0,\delta),\CC/\OOO_K).
$$
We define a conditional version of the gap measures $\mu_t$, for $t$ large enough so that $\F_t \cap \B \neq \emptyset$, by
$$
\mu_{t,\B} = \frac{1}{\card(\F_t \cap \B)} \sum_{r \in \F_t \cap \B}\Delta_{e^t d(r, \, \F_t \ssm \{r\})}.
$$
Note that $\mu_{t,\B}$ is a probability measure of tail distribution function given by
\begin{equation}\label{eq:link_Ft_mutB}
	\mu_{t,\B}(\, ]\delta, +\infty[ \,) = F_t(1,\overline{\rm D}(0,\delta),\B).
\end{equation}

\subsection{Dynamical description of the cumulative gaps distribution}
We work with the group $\SL_2(\CC)$. For every complex number $z \in \CC$ and all real numbers $t, \theta \in \RR$, we set
$$
h(z) = \begin{pmatrix} 1 & z \\ 0 & 1 \end{pmatrix}, \; a(t) = \begin{pmatrix} e^{-\frac{t}{2}} & 0 \\ 0 & e^{\frac{t}{2}} \end{pmatrix} \mbox{ and } m(\theta) = \begin{pmatrix} e^{-\frac{i\theta}{2}} & 0 \\ 0 & e^{\frac{i \theta}{2}} \end{pmatrix}.
$$
These matrices are designed so that the formula
\begin{equation}\label{eq:matrix_comput_dynamical_descr}
	m(\theta)a(t)h(z) \begin{pmatrix} p \\ q \end{pmatrix} =  \begin{pmatrix} (p+qz)e^{- \frac{t+i\theta}{2}} \\ qe^{\frac{t+i\theta}{2}} \end{pmatrix} 
\end{equation}
holds for all $p,q,z\in \CC$ and $t,\theta \in \RR$. Let $\A$ be a Borel subset of $\CC$. Identifying the elements of $\CC^2$ with their column matrices, we define the \emph{cone of $\A$} by
\begin{align*}
\CCC(\A) & = \Big\{ \begin{pmatrix} z \\ u \end{pmatrix} \in \CC \times (\overline{\rm D}(0,1)\ssm \{0\}) \, : \, z \in u \A \Big\} \subset \CC^2\\
& = \bigcup_{u\in\overline{\rm D}(0,1)\ssm \{0\}} (u\,\A) \times \{u\}.
\end{align*}
By Equation \eqref{eq:matrix_comput_dynamical_descr}, for all complex numbers $p,q,z \in \CC$ and every $t \in \RR$, we have the equivalence
\begin{equation}\label{eq:equiv_dyn_description}
	a(-t)h(-z) \begin{pmatrix} p \\ q \end{pmatrix} \in \CCC(\A) \iff \frac{p}{q} \in z + e^{-t}\A \mbox{ and } 0 < |q| \leq e^{\frac{t}{2}}.
\end{equation}
We denote by $\widehat{\OOO_K^{\; 2}}$ the set of vectors in $\OOO_K^{\; 2}$ with coprime components $p,q$ (i.e.~they satisfy $p \OOO_K + q\OOO_K = \OOO_K$). Notice that, for each $z=\tilde{z}+\OOO_K \in \CC/\OOO_K$, the set $h(-z) \widehat{\OOO_K^{\; 2}}=h(-\tilde{z}) \widehat{\OOO_K^{\; 2}}$ is well defined since $h$ is a group morphism and $\SL_2(\OOO_K)$ preserves $\widehat{\OOO_K^{\; 2}}$. Moreover, notice that a fraction $r \in \bigcup_{t\geq 0} \F_t$ (which is equal to $K$ if and only if $\OOO_K$ is principal) determines a vector in $\widehat{\OOO_K^{\; 2}}$ up to multiplication by an element of $\OOO_K^\times$. This remark and the equivalence \eqref{eq:equiv_dyn_description} grant us the following formula for the cumulative distribution function: for every integer $k \in \NN$, every bounded Borel subset $\A$ of $\CC$, every nonnegligible Borel subset $\B$ of $\CC/\OOO_K$, every real number $t \geq 0$ large enough so that $\pr_{\OOO_K}$ is injective on $e^{-t}\A$ and such that $\F_t \cap \B \neq \emptyset$, we have
\begin{equation}\label{eq:dyn_descr_Ft}
	F_t(k,\A,\B) = \frac{\card \big\{ r \in \F_t \cap \B \, : \, \card{(\CCC(\A) \cap a(-t)h(-r) \widehat{\OOO_K^{\; 2}})}=k |\OOO_K^\times| \big\}}{\card(\F_t \cap \B)}.
\end{equation}

\subsection{The asymptotics of gaps via a joint equidistribution result from homogeneous dynamics}\label{ssec:homogeneous dynamic}
Our main tool in this Section is the vague convergence result \cite[Cor.~4.2]{parkkonenpaulin2023jointfarey} for complex Farey fractions. In order to state it, we work in $G = \PSL_2(\CC)$ and denote by $\Gamma = \PSL_2(\OOO_K)$ the \emph{Bianchi group}. For a matrix $g \in \SL_2(\CC)$, we denote by $[g]=\{\pm g\}$ its projection in $G$. Define $H = [h(\CC)]$ and $M=[m(\RR)]$ which are abelian subgroups of $G$, and let $A=[a(\RR)]$ be the Cartan subgroup of $G$. We set $\Gamma_H = N_G(H)\cap\Gamma$ which can be described as followed
$$
\Gamma_H = (H \cap \Gamma) (M \cap \Gamma)= \Big\{\begin{bmatrix} a & b \\ 0 & a^{-1} \end{bmatrix} \, : \, a \in \OOO_K^\times, b \in \OOO_K\Big\}.
$$
We recall that an action $*$ of $\Gamma_H$ \emph{on the right} on $H$ is defined (from $\Gamma_H \times H$ to $H$) so that we have, for all $h \in H$ and $\gamma_1, \gamma_2 \in \Gamma_H$, the equalities $(h*\gamma_1)*\gamma_2 = h*(\gamma_1 \gamma_2)$ and $h*[I]=h$. As noticed in \cite[§~4.2]{parkkonenpaulin2023jointfarey}, the group $\Gamma_H$ acts cocompactly on the right on $H$ by the formula
$$
h' * (hm) = m^{-1} h' h m \; \mbox{ for all } h' \in H, \, h \in H \cap \Gamma \mbox{ and } m \in M \cap \Gamma.
$$
We endow the orbit space $H/\Gamma_H$ with the measure $\mu_{H/\Gamma_H}$ induced by the Haar measure $dz$ on $H$ by the branched cover $H \to H/\Gamma_H$ and normalised to be a probability measure. By the canonical identification $H/\Gamma_H = (M\cap\Gamma) \bs H / (\Gamma \cap H)$, we see $H/\Gamma_H$ as included in $M \bs G / \Gamma$. Similarly, we canonically identify $\Gamma_H \bs H$ with $(\Gamma \cap H) \bs H / (M \cap \Gamma) \subset \Gamma \bs G / M$ and we have the induced probability measure $\mu_{\Gamma_H \bs H}$ (the result \cite[Cor.~4.2]{parkkonenpaulin2023jointfarey} is stated in this setting). We have the formula $\mu_{H/\Gamma_H} = \iota _* \mu_{\Gamma_H \bs H}$ where $\iota : \Gamma g M \mapsto M g^{-1} \Gamma$ is the inversion map. We denote by $\Theta$ (resp.~$\Theta'$) the Cartan involution $M g \Gamma \mapsto M \transp{g}^{-1} \Gamma$ on $M \bs G / \Gamma$ (resp.~$\Gamma g M \mapsto \Gamma \transp{g}^{-1} M$ on $\Gamma \bs G / M$).

The definition of Farey fractions used in \cite{parkkonenpaulin2023jointfarey} is different from ours. We set $\OOO_K^{\, \prime}= \OOO_K^{\times} \ltimes \OOO_K$ with law $(a,b)(c,d) = (ac, ad+bc^{-1})$, acting on the left on $\CC$ by $((a,b),z) \mapsto a^2 z + ab$ with kernel $\{\pm (1,0)\}$. For every real number $t$, the map $z \mapsto M [a(-t)h(z)] \Gamma$ is constant on every $\OOO_K^{\, \prime}$-orbit thanks to the formula, for all $z \in \CC$ and $(a,b) \in \OOO_K^{\, \prime}$,
$$
[h(a^2z + ab)] =  \begin{bmatrix} 1 & a^2z+ab \\ 0 & 1 \end{bmatrix} = \begin{bmatrix} a & 1 \\ 0 & a^{-1} \end{bmatrix} \Big( \begin{bmatrix} 1 & z \\ 0 & 1 \end{bmatrix} \begin{bmatrix} 1 & a^{-1}b \\ 0 & 1 \end{bmatrix} \Big)\begin{bmatrix} a & 1 \\ 0 & a^{-1} \end{bmatrix}^{-1}
$$
and since $[a(-t)]$ is in the centralizer of $M$. We set, for all $t \geq 0$,
$$
\F_t' = \OOO_K^{\, \prime} \bs \big\{ \frac{p}{q} \, : \, p,q \in \OOO_K, \, p\OOO_K+q\OOO_K = \OOO_K \mbox{ and } 0 < |q| \leq e^{\frac{t}{2}} \big\} \subset \OOO_K^{\, \prime} \bs \CC.
$$
Identifying $\OOO_K$ with the subgroup ${1} \times \OOO_K$ of $\OOO_K^{\, \prime}$, we obtain a projection from $\CC/\OOO_K$ to $\OOO_K' \bs \CC$ which restricts to each set $\F_t$ onto $\F_t'$, and this restriction has fibres of order $\frac{|\OOO_K^{\; \times}|}{2}$ (except for the fibre of a finite set of points which is independent on $t$). In particular we have the estimate, as $t \to +\infty$,
\begin{equation}\label{eq:estim_card_Ft'_Ft}
	\card \, \F_t' \sim  \frac{2 \, \card \, \F_t}{|\OOO_K^\times|}.
\end{equation}
Moreover, if $D_K \neq -4, -3$, then $\OOO_K^\times = {\pm 1}$ and this canonical projection yields the identification, for every number $t\geq 0$, $\F_t' = \F_t$. Let us denote by $dx'$ the measure on $\OOO_K^{\, \prime} \bs \CC$ induced by the Lebesgue measure of $\CC$ by the branched cover $\CC \to \OOO_K^{\, \prime} \bs \CC$ and normalised to be a probability measure.

We now state a result of J.~Parkkonen and F.~Paulin about joint equidistribution of Farey fractions (in the sense of elements of $\F_t'$) and their images on an horosphere in the $3$-dimensional real hyperbolic space pushed forward by the geodesic flow. It is stated in its original version, that will need to be adapted afterwards for our purpose.

\begin{theorem}\label{th:PaPa23}{\rm ({\cite[Cor.~4.2]{parkkonenpaulin2023jointfarey}})}
	We have the vague convergence of probability measures on $(\OOO_K^{\, \prime} \bs \CC) \times 	(\Gamma \bs G / M)$, as $t \to +\infty$,
	$$
	\frac{1}{\card \, \F_t'} \sum_{r \in \F_t'} \Delta_r \otimes \Delta_{\Gamma [h(r)a(t)] M} \weakstar dx' \, \otimes \, 2 \int_{s=0}^{+\infty} \Theta'_* {[a(-s)]_* \mu_{\Gamma_H \bs H}} \; e^{-2s} \, ds.
	$$
\end{theorem}

We will prove the following version, more suitable to our study.

\begin{corollary}\label{cor:adapted_PaPa23}
	We have the narrow convergence of probability measures on the space $(\CC/\OOO_K) \times 	(M \bs G / \Gamma)$, as $t \to +\infty$,
		$$
		\frac{1}{\card \, \F_t} \sum_{r \in \F_t} \Delta_r \otimes \Delta_{M [a(-t)h(-r)] \Gamma} \rightarrow \frac{2}{\sqrt{|D_K|}} dx \, \otimes \, 2 \int_{s=0}^{+\infty} \Theta_* {[a(s)]_* \mu_{H/\Gamma_H}} \; e^{-2s} \, ds.
		$$
\end{corollary}

\begin{proof}
	First, notice that the class $M [a(-t)h(-z)] \Gamma$ is indeed well defined for every real number $t$ and every class $z \in \CC/\OOO_K$.
	
	Since $(\OOO_K^{\, \prime} \bs \CC) \times (M \bs G / \Gamma)$ is a $\sigma$-compact proper metric space, in order to obtain a narrow convergence result it is sufficient to prove a vague convergence result for which there is no loss of mass (this last point is immediate since both sides of this convergence are probability measures). Such an argument may also be made to immediately improve Theorem \ref{th:PaPa23} into a narrow convergence result.
	
	Pushing forward the result of Theorem \ref{th:PaPa23} by the continuous and proper map $\id_{\CC/\OOO_K} \times \, \iota : (r,\Gamma g M) \mapsto (r,M g^{-1} \Gamma)$ to reverse the order of quotients, by the formula $ \iota \circ \Theta' \circ [a(-s)] : \Gamma_H h \mapsto \Theta([a(s)] h^{-1} \Gamma_H) $ we obtain
	\begin{equation}\label{eq:conv_Ft'_rightorder}
	\frac{1}{\card \, \F_t'} \sum_{r \in \F_t'} \Delta_r \otimes \Delta_{M [a(-t)h(-r)] \Gamma} \weakstar dx' \, \otimes \, 2 \int_{s=0}^{+\infty} \Theta_* {[a(s)]_* \mu_{H/\Gamma_H}} \; e^{-2s} \, ds.
	\end{equation}
	It now remains to pass from a result on the family $(\F_t')_{t \geq 0}$ in $\OOO_K^{\, \prime} \bs \CC$ to one on $(\F_t)_{t\geq 0}$ in $\CC/\OOO_K$. If $D_K \neq -4, -3$, these families coincide and $dx'=\frac{2}{\sqrt{|D_K|}} dx$ hence the proof is complete. It remains to study the two cases $\OOO_K=\ZZ[i]$ and $\OOO_K=\ZZ[j]$ where $j = e^{i\frac{2\pi}{3}}$, namely the Gaussian and Eisenstein integers. The group $\OOO_K^\times$ acts on the left on $\CC/\OOO_K$ by $a \bullet \pr_{\OOO_K}(z) \mapsto \pr_{\OOO_K}(a^2 z)$ and the orbit space of this action (resp.~the restriction of this action to $\F_t$) canonically identifies with $\OOO_K^{\, \prime} \bs \CC$ (resp.~with $\F_t'$). 
	
	Let $\pr_\bullet : \CC/\OOO_K \to \OOO_K^{\, \prime} \bs \CC$ denote the orbit projection of the action $\bullet$. We want to pullback the result of Theorem \ref{th:PaPa23} by this projection. We set
	$$
	\X =
	\left\{
	\begin{array}{cl}
		\pr_{\ZZ[i]}([0,1]\cup[1,i]\cup[i,0]) & \mbox{ if } \OOO_K=\ZZ[i], \vspace{2mm}
		\\ \pr_{\ZZ[j]}([0,1]\cup[0,j]\cup[\frac{1}{2}, \frac{j^2}{j^2-1}] &
		\\ \cup[\frac{j^2}{j^2-1},1+\frac{j}{2}]\cup[\frac{j}{2}, \frac{j^2}{j-1}]\cup[\frac{j^2}{j-1}, \frac{1}{2}+j]) & \mbox{ if } \OOO_K=\ZZ[j].
	\end{array}
	\right.
	$$
	We verify that all orbits $\OOO_K^\times \bullet z \in \pr_\bullet((\CC/\OOO_K)\ssm \X)$ have the same cardinality $\frac{|\OOO_K^\times|}{2}$. Let us make $\frac{|\OOO_K^\times|}{2}$ choices for each orbit, allowing us to write a decomposition
	\begin{equation}\label{eq:decompos_COK_COK'}
		(\CC/\OOO_K)\ssm \X = \bigcup_{i=1}^{|\OOO_K^\times| / 2} \X_i
	\end{equation}
	where each $\X_i$ is an open subset of $\CC/\OOO_K$ such that $\pr_\bullet{_{|\X_i}}$ is an homeomorphism onto $(\OOO_K^{\, \prime} \bs \CC)\ssm \pr_\bullet(\X)$. See Figure \ref{fig:decompo_tores_zizj} for the choice of such a decomposition.
	\begin{figure}[ht]
		\centering
		\includegraphics[height=7.7cm]{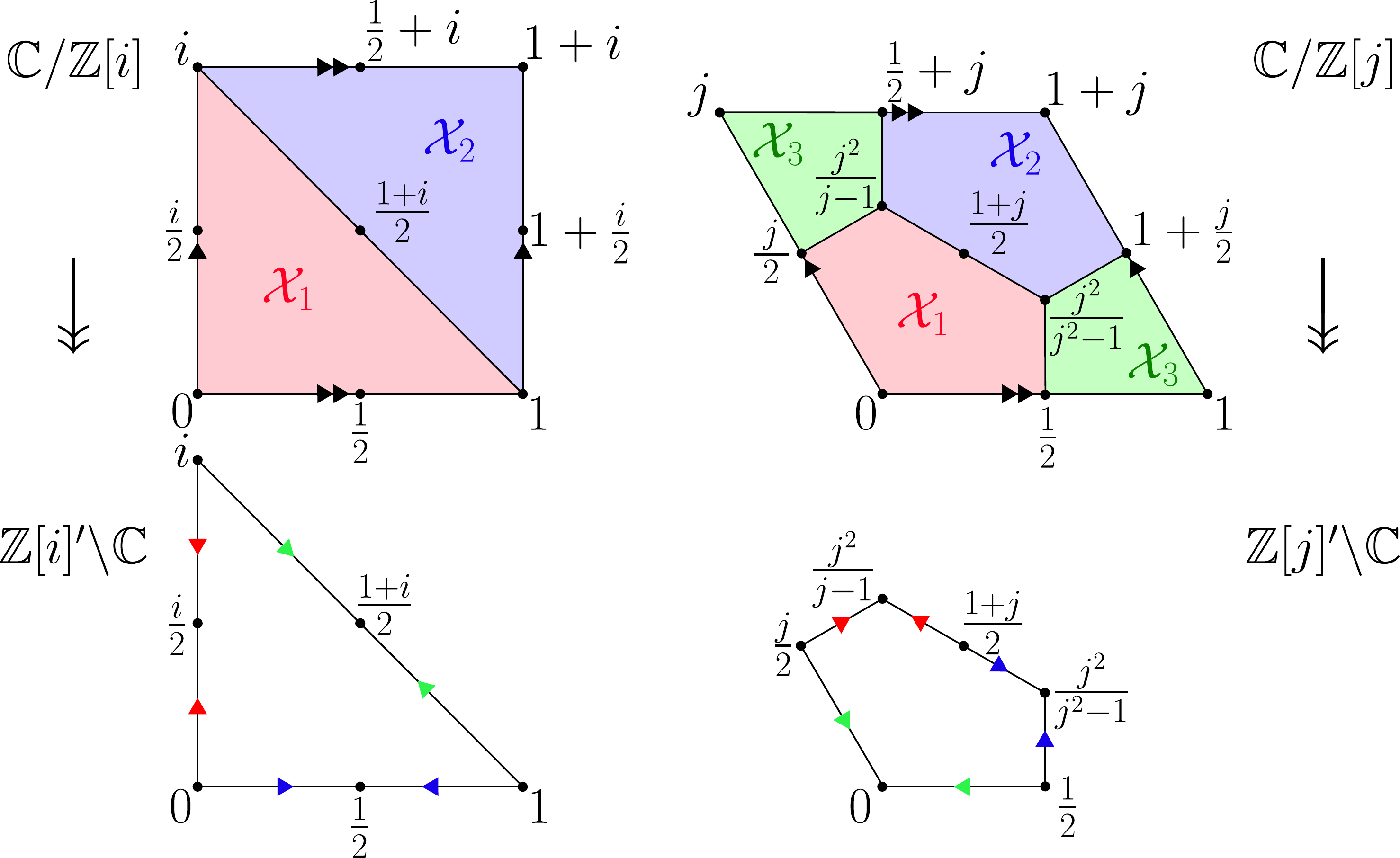}
		\caption{An example of a decomposition $\{\X_1, \X_2\}$ of $(\CC/\ZZ[i]) \ssm \X$ (resp.~$\{\X_1,\X_2,\X_3\}$ of $(\CC/\ZZ[j]) \ssm \X$) into parts which are each homeomorphic to $(\ZZ[i]' \bs \CC)\ssm \pr_\bullet(\X)$ (resp.~to $(\ZZ[j]' \bs \CC)\ssm \pr_\bullet(\X)$). The sets $\X$ and $\pr_\bullet(\X)$ are represented by black lines and black dots.}
		\label{fig:decompo_tores_zizj}
	\end{figure}
	
	Since the map $z \in \CC/\OOO_K \mapsto M[a(-t)h(-z)]\Gamma$ is constant on each orbit of $\bullet$, since $\X$ is negligible and since we have the formula $(\pr_\bullet{_{|\X_i}}^{-1})_* dx' = \frac{|\OOO_K^\times|}{\sqrt{|D_K|}} dx_{|\X_i}$, pulling back the convergence \eqref{eq:conv_Ft'_rightorder} by the map $(\pr_\bullet{_{|\X_i}}, \id_{M\bs G / \Gamma})$ grants us
	$$
	\frac{1}{\card \, \F_t'} \sum_{r \in \F_t \cap \X_i} \Delta_r \otimes \Delta_{M [a(-t)h(-r)] \Gamma} \weakstar \frac{|\OOO_K^\times|}{\sqrt{|D_K|}} dx_{|\X_i} \, \otimes \, 2 \int_{s=0}^{+\infty} \Theta_* {[a(s)]_* \mu_{H/\Gamma_H}} \; e^{-2s} \, ds.
	$$
	Combining the latter convergence with the estimate \eqref{eq:estim_card_Ft'_Ft}, summing over the decomposition \eqref{eq:decompos_COK_COK'}, and using the remark above for passing from vague to narrow convergence, this concludes the proof.
\end{proof}

Let $k \in \NN$, fix $\A$ a bounded rotation-invariant Borel subset of $\CC$, as well as a nonnegligible Borel subset $\B$ of $\CC/\OOO_K$. Since $\A$ is rotation-invariant, we have, for every real number $\theta \in \RR$, the equality $m(\theta) \CCC(\A) = \CCC(e^{i\theta} \A) = \CCC(\A)$. Thus, for every matrix $g \in \SL_2(\CC)$, the cardinality of $\CCC(\A) \cap g \widehat{\OOO_K^{\; 2}}$ only depends on the class $M[g]\Gamma$ of $g$ in $M \bs G / \Gamma$. Hence we can define the following Borel subset of $(\CC/\OOO_K) \times 	(M \bs G / \Gamma)$,
\begin{equation}\label{eq:def_Ck}
	\C_k = \B \times \Y_k \mbox{ where } \Y_k = \{ M[g]\Gamma \in M \bs G / \Gamma \, : \, \card(\CCC(\A) \cap g \widehat{\OOO_K^{\; 2}}) = k |\OOO_K^\times| \}.
\end{equation}
And using Equation \eqref{eq:dyn_descr_Ft} we obtain the formula, for $t$ large enough so that $\F_t \cap \B \neq \emptyset$,
\begin{equation}\label{eq:link_Ft'_thPapa}
	F_t(k,\A,\B) = \frac{\card \, \F_t}{\card(\F_t \cap \B)} \frac{1}{\card \, \F_t} \big(\sum_{r \in \F_t} \Delta_r \otimes \Delta_{M [a(-t)h(-r)] \Gamma} \big) (\C_k).
\end{equation}
Assuming that $\B$ has a negligible boundary, since $\B$ has nonzero measure, the equidistribution of $(\F_t)_{t \geq 0}$ implies the convergence,
\begin{equation}\label{eq:conv_card_FtB}
	\frac{\card \, \F_t}{\card(\F_t \cap \B)} \underset{t\to+\infty}{\longrightarrow} \frac{\sqrt{|D_K|}}{2 \, dx(\B)}.
\end{equation}
The behaviour of the remaining part in Equation \eqref{eq:link_Ft'_thPapa} will be described by the narrow convergence property in Corollary \ref{cor:adapted_PaPa23}. For this purpose, we need to control the limit measure of the boundary of $\C_k$.

\begin{lemma}\label{lem:boundary_Ck}
	Assume that the boundaries $\partial \A$ and $\partial \B$ are negligible respectively for the Lebesgue measure on $\CC$ and for the measure $dx$. Then
	$$
	\Big(\frac{2}{\sqrt{|D_K|}} dx \, \otimes \, 2 \int_{s=0}^{+\infty} \Theta_* {[a(s)]_* \mu_{H/\Gamma_H}} \; e^{-2s} \, ds \Big) (\partial \C_k) = 0.
	$$
\end{lemma}
\begin{proof}
	By the definition \eqref{eq:def_Ck}, we have the equality $\partial \C_k = (\partial \B \times \overline{\Y_k}) \cup (\overline{\B} \times \partial \Y_k)$. The first term in this union has measure $0$ since $dx(\partial \B) = 0$. For the second term, since $[a(-s)]$ and $\Theta=\Theta^{-1}$ are homeomorphisms of $M \bs G / \Gamma$, we obtain
	\begin{align*}
		\int_{s=0}^{+\infty} \Theta_* {[a(s)]_* \mu_{H/\Gamma_H}} (\partial \Y_k) \; e^{-2s} \, ds & = \int_0^{+\infty} \mu_{H/\Gamma_H} ([a(-s)]\Theta (\partial \Y_k)) \; e^{-2s} \, ds
		\\ & = \int_0^{+\infty} \mu_{H/\Gamma_H} (\partial [a(-s)]\Theta(\Y_k)) \; e^{-2s} \, ds.
	\end{align*}
	For every $p,q \in \CC$, we compute
	\begin{equation}\label{eq:boundary_lift_description}
		a(-s) \transp{h(-z)} \begin{pmatrix} p \\ q \end{pmatrix} = \begin{pmatrix} p e^{\frac{s}{2}} \\ (q-z p) e^{-\frac{s}{2}} \end{pmatrix}.
	\end{equation}
	In this proof, we will frequently use the following description of $\partial \CCC(\A)$, arising from its definition,
	\begin{equation}\label{eq:boundary_coneA}
		\partial \CCC(\A) = \bigcup_{u \in \overline{\rm D}(0,1)} (u \, \partial \A) \times \{u\} \; \cup \bigcup_{u \in {\rm C}(0,1)} (u \, \overline{\A}) \times \{u\},
	\end{equation}
	where ${\rm C}(0,1)$ denotes the unit circle of $\CC$. We fix a real number $s > 0$. The intersection
	\begin{equation}\label{eq:before_boundary_intersection}
		[a(-s)]\Theta(\Y_k) \cap (H / \Gamma_H) = \{ [h] * \Gamma_H \, : \, \card(\CCC(\A) \cap a(-s) \transp{h}^{-1} \widehat{\OOO_K^{\; 2}}) = k |\OOO_K^\times| \}
	\end{equation}
	has its boundary contained in $\{ [h] * \Gamma_H \, : \, \partial \CCC(\A) \cap a(-s) \transp{h}^{-1} \widehat{\OOO_K^{\; 2}} \neq \{0\}  \}$. Now notice that the points $a(-s) \transp{h}^{-1} \begin{pmatrix} 0 \\ q\end{pmatrix}$ given by the computation \eqref{eq:boundary_lift_description} do not depend on the choice of $h \in H$. Hence each of these points is, or is not, in $\CCC(\A)$, independently on the choice of $h$. Thus the boundary of the right-hand side of Equation \eqref{eq:before_boundary_intersection} is furthermore contained in
	$$
	Z=\{ [h] * \Gamma_H \, : \, \partial \CCC(\A) \cap a(-s) \transp{h}^{-1} (\widehat{\OOO_K^{\; 2}} \ssm (\{0\} \times \CC)) \neq \{0\}  \}.
	$$
	The preimage of $Z$ in $H$ is
	$$\tilde{Z} = \{ h(z) \, : \, z \in \CC, \, \partial \CCC(\A) \cap a(-s) \transp{h(-z)} (\widehat{\OOO_K^{\; 2}} \ssm (\{0\} \times \CC)) \neq \{0\} \}.$$
	Let us investigate this subset more thoroughly. Let $z \in \CC$ be such that $h(z) \in \tilde{Z}$. Let $(p,q) \in \OOO_K$ with $p \neq 0$ be such that
	$$
	a(-s) \transp{h(-z)} \begin{pmatrix} p \\ q \end{pmatrix} = \begin{pmatrix} p e^{\frac{s}{2}} \\ (q-z p) e^{-\frac{s}{2}} \end{pmatrix} \in \partial \CCC(\A).
	$$
	Using the description \eqref{eq:boundary_coneA} and since $p \neq 0$, we have:
	\begin{itemize}
		\item either $pe^{\frac{s}{2}} \in u \partial \A$ and $(q-z p) e^{-\frac{s}{2}} = u$ (and $u \in \overline{\rm D}(0,1) \ssm \{0\}$), which implies that $z = \frac{q}{p} - \frac{u e^s}{pe^{\frac{s}{2}}} \in \frac{q}{p} - e^{s} (\partial \A \ssm \{0\})^{-1}$,
		\item or $(q-z p) e^{-\frac{s}{2}} = u \in {\rm C}(0,1)$ (and $pe^{\frac{s}{2}} \in u \A$), which implies in particular that $z = \frac{q}{p} - \frac{u e^\frac{s}{2}}{p} \in \frac{q}{p} - \frac{1}{p}e^{\frac{s}{2}} {\rm C}(0,1)$.
	\end{itemize}
	In both cases, $z$ belongs to a negligible Borel set depending only on $p,q$ (and of course $\A$ and $s$). Since $\OOO_K^{\; 2}$ is countable, this implies that $dz(\tilde{Z})=0$ (where $dz$ is the Lebesgue measure on $H$), and then using the definition of $\mu_{H/\Gamma_H}$, we finally obtain the equality $\mu_{H/\Gamma_H} (\partial [a(-s)]\Theta(\Y_k))=0$.
\end{proof}

We now combine all the previous results of this Section to obtain a convergence result for the the tail functions $F_t$.
\begin{lemma}\label{lem:conv_Ft}
	For every $k \in \NN$, for every bounded rotation-invariant Borel subset $\A$ of $\CC$ with negligible boundary for the Lebesgue measure, for every nonnegligible Borel subset $\B$ of $\CC/\OOO_K$ with negligible boundary for the measure $dx$, we have the convergence, as $t \to +\infty$,
	{\small
	$$
	F_t(k,\A,\B) \rightarrow 2 \int_{s=0}^{+\infty} \mu_{H/\Gamma_H}(\{ [h] * \Gamma_H \, : \, \card(\CCC(\A) \cap a(-s) \transp{h}^{-1} \widehat{\OOO_K^{\; 2}}) = k |\OOO_K^\times| \}) \, e^{-2s} \, ds.
	$$
	}
\end{lemma}

\begin{proof}
	Using the definition \eqref{eq:def_Ck} of $\C_k$ and the fact that its boundary $\partial \C_k$ is negligible for the limit measure in Corollary \ref{cor:adapted_PaPa23}, we obtain the convergence, as $t \to +\infty$,
	{\footnotesize
	\begin{align*}
	\frac{1}{\card \, \F_t} \big(\sum_{r \in \F_t} \Delta_r \otimes \Delta_{M [a(-t)h(-r)] \Gamma} \big) (\C_k)
 \to & \Big(\frac{2}{\sqrt{|D_K|}} dx \, \otimes \, 2 \int_{s=0}^{+\infty} \Theta_* {[a(s)]_* \mu_{H/\Gamma_H}} \; e^{-2s} \, ds \Big) (\C_k)
	\\ & = \frac{4 dx(\B)}{\sqrt{|D_K|}} \int_{s=0}^{+\infty} (\Theta_* {[a(s)]_* \mu_{H/\Gamma_H}})(\Y_k) \; e^{-2s} \, ds. \numberthis\label{eq:conv_rightpart_Ft}
	\end{align*}
	}Then for all $s \geq 0$ and $g \in \SL_2(\CC)$, using the change of variable $h = a(-s) \transp{g}^{-1}$ i.e.~$g = a(-s) \transp{h}^{-1}$ for the last equality, we obtain
	\begin{align*}
	& (\Theta_* {[a(s)]_* \mu_{H/\Gamma_H}})(\Y_k) = \mu_{H/\Gamma_H}([a(-s)] \Theta(\Y_k))
	\\= \; & \mu_{H/\Gamma_H}(\{ [a(-s)\transp{g}^{-1}] * \Gamma_H \, : \, g \in \SL_2(\CC), \, [a(-s)\transp{g}^{-1}]\in H
	\\ & \hspace{4.7cm} \mbox{ and }\, \card(\CCC(\A) \cap g \widehat{\OOO_K^{\; 2}}) = k |\OOO_K^\times| \})
	\\ = \; & \mu_{H/\Gamma_H}(\{ [h] * \Gamma_H \, : \, \card(\CCC(\A) \cap a(-s) \transp{h}^{-1} \widehat{\OOO_K^{\; 2}}) = k |\OOO_K^\times| \}) \, e^{-2s} \, ds.\numberthis\label{eq:formula_limitmeasure_Y_k}
	\end{align*}
	Combining the formula \eqref{eq:link_Ft'_thPapa} for $F_t(k,\A,\B)$ first with the convergence \eqref{eq:conv_card_FtB} then with Equations \eqref{eq:conv_rightpart_Ft} and \eqref{eq:formula_limitmeasure_Y_k}, the proof of the lemma is complete.
\end{proof}

We now state and prove the main theorem in this section.

\begin{theorem}\label{th:existence_density}
	There exists a probability measure $\mu$ on $[0,+\infty[\,$, absolutely continuous with respect to the Lebesgue measure, such that for every nonnegligible Borel subset $\B$ of $\CC/\OOO_K$ with negligible boundary, we have the narrow convergence, as $t \to +\infty$,
	$$
	\mu_{t,\B} \to \mu.
	$$
	Furthermore, we have the equality $\mu([0,1])=0$ and the cumulative distribution function of $\mu$ is given by the formula, for every real number $\delta \geq 0$,
	{\small
	$$
	\mu([0,\delta]) = 1-\frac{4}{\sqrt{|D_K|}} \int_{s=0}^{+\infty} dx(\{ z \, : \, \card(\CCC(\overline{\rm D}(0,\delta)) \cap a(-s) \transp{h(z)} \widehat{\OOO_K^{\; 2}}) = |\OOO_K^\times| \}) \, e^{-2s}ds.
	$$
	}
\end{theorem}

\begin{proof}
	For every real number $\delta \geq 0$, set
	$$
	F(\delta)=1-\frac{4}{\sqrt{|D_K|}} \int_{s=0}^{+\infty} dx(\{ z \, : \, \card(\CCC(\overline{\rm D}(0,\delta)) \cap a(-s) \transp{h(z)} \widehat{\OOO_K^{\; 2}}) = |\OOO_K^\times| \}) \, e^{-2s}ds.
	$$
	Recalling the definition $\mu_{H/\Gamma_H} = (z \mapsto [h(z)]*\Gamma_H)_* (\frac{2}{\sqrt{|D_K|}} dx)$, by Equation \eqref{eq:link_Ft_mutB} and Lemma \ref{lem:conv_Ft} with $k=1$ and $\A=\overline{\rm D}(0,\delta)$, we obtain the convergence, for every $\delta \geq 0$,
	$$
	\mu_{t,\B}([0,\delta]) \underset{t \to +\infty}{\longrightarrow} F(\delta).
	$$
	Since $(\mu_{t,\B})_{t \geq 0}$ is a sequence of probability measures, it remains to prove that $F$ is the cumulative distribution function of some probability measure $\mu$ on $\,[0,+\infty[ \,$ with density with respect to the Lebesgue measure, i.e.~that $F$ is nondecreasing, is absolutely continuous, verifies the equality $F(0)=0$ and the convergence $F(\delta) \to 1$ as $\delta \to +\infty$.
		
	We begin by investigating more thoroughly the definition of $F$. For every $\delta \geq 0$, we have the description
	$$
	\CCC(\overline{\rm D}(0,\delta)) = \bigcup_{u \in \, ]0,1]} \overline{\rm D}(0,u\delta) \times {\rm C}(0,u).
	$$
	For all $p,q \in \widehat{\OOO_K^{\; 2}}$, all real numbers $\delta, s \geq 0$, and every complex number $z \in \CC$, this yields the following equivalence
	\begin{equation}\label{eq:equiv_inter_defF}
		a(-s) \transp{h(z)} \begin{pmatrix} p\\-q \end{pmatrix} \in \CCC(\overline{\rm D}(0,\delta)) \iff
		\left\{
		\begin{array}{c}
			|p| \leq \delta e^{-s} |pz-q| \, ,\\
			0 < |pz-q| \leq e^{\frac{s}{2}}.
		\end{array}
		\right.
	\end{equation}
	
	The equality $F(0)=0$ is then a consequence of Equation \eqref{eq:equiv_inter_defF} with $\delta=0$, since then $q$ may take any value in $\OOO_K^\times$ (and only these ones, in order for $0$ and $q$ to be coprime), no matter what are the values of $z \in \CC$ and $s \geq 0$. Hence we obtain $F(0) = 1-\frac{4}{\sqrt{|D_K|}} \int_0^{+\infty} \frac{\sqrt{|D_K|}}{2} \, e^{-2s} ds = 1-\frac{2}{2}=0.$

	Furthermore, the latter argument proves that for any Borel set $\A$ containing $0$, the intersection $\CCC(\A) \cap a(-s) \transp{h(z)} \widehat{\OOO_K^{\; 2}}$ contains at least $|\OOO_K^\times|$ points. Since in addition the cardinality of this intersection is nondecreasing with respect to $\delta$ for $\A=\overline{\rm D}(0,\delta)$, the function $F$ is nondecreasing as well.
	
	In order to compute the limit of $F$ at $+\infty$, we need to describe points in the intersection $\CCC(\overline{\rm D}(0,\delta)) \cap a(-s) \transp{h(z)} \widehat{\OOO_K^{\; 2}}$ which are nontrivial, namely not of the form $a(-s) \transp{h(z)} \begin{pmatrix} 0 \\ q \end{pmatrix}$ with $q \in \OOO_K^\times$. Each of these points necessarily has a nonzero first coordinate $pe^\frac{s}{2}$. The equivalence \eqref{eq:equiv_inter_defF} then translates to, for all $s \geq 0$ and $\delta >0$,
	\begin{align*}
			\card(\CCC(\overline{\rm D}(0,\delta)) \cap a(-s) \transp{h(z)} \widehat{\OOO_K^{\; 2}}) \neq |\OOO_K^\times| \iff & \exists (p,q) \in \widehat{\OOO_K^{\; 2}}, p \neq 0,
			\\ & \frac{e^s}{\delta} \leq \big|z-\frac{q}{p}\big| \leq \frac{e^\frac{s}{2}}{|p|}. \numberthis\label{eq:equiv_inter_defF_pnonzero}
	\end{align*}
	Using the pigeonhole principle, we have the following adapted version of Dirichlet Diophantine approximation theorem \cite[§~7,Prop.~2.7]{elstrod&co1998groupsactinghypspace}:
	there exists a constant $\kappa>0$ (depending on $K$) such that for every nonrational complex point $z \in \CC \ssm K$, there exist infinitely many pairs $(p,q) \in \widehat{\OOO_K^{\; 2}}$ with $p \neq 0$ verifying
	\begin{equation}\label{eq:dirichlet_complex_th}
		\big|z-\frac{q}{p}\big| \leq \frac{\kappa}{|p|^2}
	\end{equation}

	In particular, for each point $z \in \CC \ssm K$ there exists such a pair $(p,q)$ with $|p| \geq \kappa$, hence the right-hand inequality of Equation \eqref{eq:equiv_inter_defF_pnonzero} holds for this pair. Since $z \notin K$, we have $|z - \frac{q}{p}| >0$, hence the equivalence \eqref{eq:equiv_inter_defF_pnonzero} yields, for $\delta$ large enough (depending on $z$),
	$$
	a(-s) \transp{h(z)} \begin{pmatrix} p\\-q \end{pmatrix} \in \CCC(\overline{\rm D}(0,\delta)).
	$$
	This gives a nontrivial point in the intersection defining $F$. Since $\CC \ssm K$ has full measure, Beppo Levi's monotone convergence lemma gives, for every $s \geq 0$, as $\delta \to +\infty$,
	$$
	dx(\{ z \in \CC/\OOO_K \, : \, \card(\CCC(\overline{\rm D}(0,\delta)) \cap a(-s) \transp{h(z)} \widehat{\OOO_K^{\; 2}}) = |\OOO_K^\times| \}) \longrightarrow 0.
	$$
	By the dominated convergence theorem, since the integrand defining $F$ is bounded from above by the $L^1$ function $s \mapsto \frac{4}{\sqrt{|D_K|}}e^{-2s}$ on $[0,+\infty[\,$, we finally have
	$$
	F(\delta) \underset{\delta \to +\infty}{\longrightarrow} 1.
	$$
	
	The last property to check is the absolute continuity of $F$. For every point $z \in \CC$ and all real numbers $r, R \geq 0$, recall that $A(z,r,R)$ is the closed annulus in $\CC$ centred at $z$ and of radii $r$ and $R$ (hence empty if $r > R$). Let $s \geq 0$. By the equivalence \eqref{eq:equiv_inter_defF_pnonzero}, we obtain, for every real number $\delta > 0$ (and using the change of variable $(p,q) \mapsto (p,-q)$ on $\widehat{\OOO_K^{\; 2}}$),
	\begin{align*}
	& dx(\{ z \, : \, \card\big(\CCC(\overline{\rm D}(0,\delta)) \cap a(-s) \transp{h(z)} \widehat{\OOO_K^{\; 2}}\big) = |\OOO_K^\times|\}) \numberthis\label{eq:formula_integrand_F_annuli}
	\\ = \; & \frac{\sqrt{|D_K|}}{2}
	- dx\Big(
	\bigcup_{\substack{(p,q) \in \widehat{\OOO_K^{\; 2}} \\ p \neq 0}} \pr_{\OOO_K}\Big( A\Big(\frac{q}{p}, \frac{e^s}{\delta}, \frac{e^\frac{s}{2}}{|p|}\Big)
	\Big) \Big).
	\end{align*}
	Let us fix a real number $M >0$. We now show that this union is finite in a uniform way: there exists a finite subset $S_M$ of $\widehat{\OOO_K^{\; 2}}$ such that for all $\delta \in \, ]0,M]$ and $s\geq 0$, we have
	\begin{equation}\label{eq:union_annuli_finite}
		\bigcup_{\substack{(p,q) \in \widehat{\OOO_K^{\; 2}} \\ p \neq 0}} \pr_{\OOO_K}\Big( A\Big(\frac{q}{p}, \frac{e^s}{\delta}, \frac{e^\frac{s}{2}}{|p|}\Big) \Big)  = \bigcup_{(p,q)\in S_M} \pr_{\OOO_K}\Big( A\Big(\frac{q}{p}, \frac{e^s}{\delta}, \frac{e^\frac{s}{2}}{|p|}\Big) \Big).
	\end{equation}
	For a fixed $\delta > 0$, in the left-hand side of Equation \eqref{eq:union_annuli_finite}, most of the annuli are empty, namely those associated to pairs $(p,q)$ verifying $|p| > \delta e^{-\frac{s}{2}}$. Hence we may restrict this union to indices $p,q \in \OOO_K$ with $0 < |p| \leq M$. This describe only a finite set for $p$. Now, for each such point $p$, the equality $\diam(\CC/p\OOO_K)=|p|\diam(\CC/\OOO_K)$ implies that the map $q \mapsto \pr_{\OOO_K}\big(\frac{q}{p}\big)$ is surjective from $\overline{\rm D}(0,|p|\diam(\CC/\OOO_K))\cap\OOO_K$ to $ \pr_{\OOO_K}(\frac{1}{p}\OOO_K)$. Thus, Equation \eqref{eq:union_annuli_finite} holds for the finite set
	$$
	S_M = \{ (p,q) \in \widehat{\OOO_K^{\; 2}} \, : \, 0<|p|<M \mbox{ and } |q| \leq |p| \diam(\CC/\OOO_K) \}.
	$$
	By Equations \eqref{eq:formula_integrand_F_annuli} and \eqref{eq:union_annuli_finite}, for every $s \geq 0$, the function
	$$
	f_s : \delta \mapsto 1 - \frac{2}{\sqrt{|D_K|}} dx\big(\{ z \, : \, \card(\CCC(\overline{\rm D}(0,\delta)) \cap a(-s) \transp{h(z)} \widehat{\OOO_K^{\; 2}})=|\OOO_K^\times|\}\big)
	$$
	is absolutely continuous over $\,]0,+\infty[\,$.
	It is also continuous at $0$ since it is constant equal to $0$ on $[0,e^{\frac{s}{2}}]$ (this is a direct consequence of Equation \eqref{eq:formula_integrand_F_annuli} and the fact that $\sys_{\OOO_K}=1$). This implies that the measure $\mu$ verifies the level repulsion equality $\mu([0,1])=F(1)=0$. Furthermore, since $f_s$ is nondecreasing, its derivative $f_s'$ is defined and nonnegative almost everywhere. And by Equations \eqref{eq:formula_integrand_F_annuli} and \eqref{eq:union_annuli_finite}, the function $(s,\delta) \mapsto f_s(\delta)$ is continuous (hence measurable) on $[0,+\infty[ \, \times [0,+\infty[\,$. Thus, by Fubini's integration theorem we obtain, for every $\delta_1, \delta_2 \geq 0$,
	\begin{align*}
	F(\delta_2) - F(\delta_1) & = 2 \int_{s=0}^{+\infty} (f_s(\delta_2) - f_s(\delta_1)) \, e^{-2s} ds
	\\ & = 2 \int_{\delta=\delta_1}^{\delta_2} \int_{s=0}^{+\infty} f_s'(\delta) \, e^{-2s} ds \, d\delta
	\end{align*}
	Hence $F$ is absolutely continuous, with almost everywhere derivative given by the formula $F'(\delta) = 2 \int_{s=0}^{+\infty} f_s'(\delta) \, e^{-2s} ds$.
\end{proof}

\section{Tails and details}
In this section, we present an explicit tail estimate and present it in this section. For all real numbers $s,\delta >0$, we set once again $f_s(\delta) = \frac{2}{\sqrt{|D_K|}} dx\big(
\bigcup_{\substack{(p,q) \in \widehat{\OOO_K^{\; 2}} \\ p \neq 0}} \pr_{\OOO_K}\big( A\big(\frac{q}{p}, \frac{e^s}{\delta}, \frac{e^\frac{s}{2}}{|p|}\big)
\big) \big)$. Theorem \ref{th:existence_density} and Equations \eqref{eq:formula_integrand_F_annuli} and \eqref{eq:union_annuli_finite} yield the following description of $\mu$.

\begin{corollary}
	For every $\delta > 0$, we have the formula
	$$
	\mu([0,\delta])= 2 \int_{s=0}^{+\infty} f_s(\delta) e^{-2s} \, ds.
	$$
	And for each value of $\delta \geq 0$, the union defining $f_s(\delta)$ can be restricted to a finite set of indices $(p,q)$ without changing the whole union (hence without changing the value $f_s(\delta)$), in a uniform way over $\delta$ varying in any bounded interval and $s$ varying in $[0,+\infty[\,$.
\end{corollary}

The following lemma begins the study of the functions $f_s$. We use the standard notation
$$
\omega_K=\left\{
\begin{array}{cl}
	i \frac{\sqrt{|D_K|}}{2} & \mbox{ if } D_K=0 \mod 4, \vspace{2mm}
	\\ \frac{1+i\sqrt{|D_K|}}{2} & \mbox{ if } D_K=1 \mod 4,
\end{array} \right.
$$
so that $\OOO_K=\ZZ[\omega_K]$. We use the notation $\ln$ for the natural logarithm function.
\begin{lemma}\label{lem:begin_study_fsdelta}
	\begin{itemize}
		\item[$\bullet$] Let
		$$
		s_\pm = 2 \ln\big( \frac{\delta}{2} \pm \frac{\delta}{2} \sqrt{1 - \frac{4|\omega_K|}{\delta}} \big).
		$$
		If $s,\delta$ verify the inequalities $\delta \geq 4|\omega_K|$ and $s_- \leq s \leq s_+$, then we have the equality $f_s(\delta)=1$. The same equality holds if we replace the assumptions on $s,\delta$ by $2 \ln|\omega_K| \leq s \leq \ln(\frac{\delta}{2})$.
		\item[$\bullet$] If $s \geq 2 \ln(\delta)$, then $f_s(\delta)=0$.
		\item[$\bullet$] Set $c_K = \min\{|z| \, : \, z \in \OOO_K, |z| > 1\}$. If $s \geq 2 \ln(\frac{|1+\omega_K| c_K}{c_K-1})$, then the expression of $f_s(\delta)$ can be simplified as follows:
		$$
		f_s(\delta) = \frac{2}{\sqrt{|D_K|}} dx\big(\pr_{\OOO_K}\big( A\big(0, \frac{e^s}{\delta}, e^\frac{s}{2}\big)
		\big) \big).
		$$
	\end{itemize}
\end{lemma}

\begin{proof}
	\begin{itemize}
		\item[$\bullet$] We assume firt that $\delta \geq 4|\omega_K|$ and $s_- \leq s \leq s_+$. Notice that $s_\pm$ are the zeros of the function $s \mapsto e^{\frac{s}{2}} - \frac{e^s}{\delta} - |\omega_K|$. Thus, for $s \in [s_-,s_+]$, the annulus $A(0,\frac{e^s}{\delta},e^{\frac{s}{2}})$ has a difference of radii greater than (or equal to) $|\omega_K|$, allowing it to project onto the whole torus $\CC/\OOO_K$ (using an appropriate translate of the fundamental domain $[0,1]+[0,1]\omega_K$ and the decrease of curvature of a disk when its radius increases). By definition, we then obtain $f_s(\delta)=1$.
		
		If we assume $2 \ln|\omega_K| < s < \ln(\frac{\delta}{2})$ instead, then the consequential inclusion of annuli $A(0,\frac{1}{2}, |\omega_K|) \subset A(0,\frac{e^s}{\delta},e^{\frac{s}{2}})$ and the fact that the annulus $A(0,\frac{1}{2}, |\omega_K|)$ projects onto the whole torus $\CC/\OOO_K$ (using the four translates $A(z,\frac{1}{2},|\omega_K|)$ centred at $z=0, 1, \omega_K$ and $1+\omega_K$ and the fundamental domain $[0,1]+[0,1]\omega_K$) prove that we still have $f_s(\delta)=1$.

		\item[$\bullet$] If $s \geq 2\ln(\delta)$, then for every point $p \in \OOO_K \ssm \{0\}$, we have $\frac{e^s}{\delta} \geq e^\frac{s}{2} \geq \frac{e^\frac{s}{2}}{|p|}$ so that each annulus $A(\frac{q}{p}, \frac{e^s}{\delta}, \frac{e^\frac{s}{2}}{|p|})$ has measure zero (it is empty or equal to a circle).
		
		\item[$\bullet$] Let $p,q \in \OOO_K$ be such that $p \neq 0$ and let $z \in A(\frac{q}{p}, \frac{e^s}{\delta},\frac{e^\frac{s}{2}}{|p|})$. We want to prove that $\pr_{\OOO_K}(z) \in \pr_{\OOO_K}(A(0,\frac{e^s}{\delta}, e^\frac{s}{2}))$ or equivalently that there exists $q' \in \OOO_K$ such that $q'\in A(z,\frac{e^s}{\delta}, e^\frac{s}{2})$. If $p \in \OOO_K^\times$, the latter already holds with $q'=qp^{-1}$. Otherwise, first notice that we have the inequality $|p| \geq c_K$, hence $\frac{e^\frac{s}{2}}{c_K} \geq \frac{e^s}{\delta}$ since $A(\frac{q}{p}, \frac{e^s}{\delta},\frac{e^\frac{s}{2}}{|p|})$ is nonempty. In addition, the annulus $A(z,\frac{e^s}{\delta}, e^\frac{s}{2})$ contains a disk of diameter $L=e^\frac{s}{2}-\frac{e^s}{\delta} \geq e^\frac{s}{2}(1-\frac{1}{c_K})$. If $L \geq |1+\omega_K|$ ($= \diam_{\OOO_K}$), then this disk contains a fundamental domain of $\OOO_K$ (by convexity), thus contains an integer point $q' \in \OOO_K$. Thus, the condition $e^\frac{s}{2}(1-\frac{1}{c_K}) \geq |1+\omega_K|$ is sufficient to simplify the expression of $f_s(\delta)$ as stated in the lemma.
	\end{itemize}
\vspace{-0.5cm}
\end{proof}

\begin{remark}{\rm
		The third point of Lemma \ref{lem:begin_study_fsdelta} will not be used in this paper, as we will focus on the asymptotic gap distribution tail, and for that purpose the first two points of this lemma will be sufficient. Nevertheless, we believe it could be a useful tool for a later attempt at studying more thoroughly this distribution. The explicit value of $c_K$ in this third point is simply given by
		$$
		c_K=
		\left\{
		\begin{array}{cl}
			\sqrt{2} & \mbox{ if } \OOO_K=\ZZ[i], \ZZ[i\sqrt{2}] \mbox{ or } \ZZ[\frac{1+i\sqrt{7}}{2}] \quad (\mbox{i.e.~} D_K = -4, -8 \mbox{ or } -7),
			\\ \sqrt{3} & \mbox{ if } \OOO_K=\ZZ[j] \mbox{ or } \ZZ[\frac{1+i\sqrt{11}}{2}] \quad (\mbox{i.e.~} D_K = -3 \mbox{ or } -11),
			\\ 2 & \mbox{ otherwise } \quad (\mbox{i.e.~} D_K \leq -15).
		\end{array}
		\right.
		$$
	}
\end{remark}

The following corollary states that the asymptotic distribution for gaps in complex Farey fractions is \emph{heavy tailed} (the tail distribution decays at less than exponential speed).
\begin{corollary}\label{cor:tail}
	As $\delta$ goes to $+\infty$, we have the estimate
	$$
	\frac{1}{\delta^4} \leq \mu(\, ]\delta,+\infty[\,) - 2 \int_{0}^{2\ln|\omega_K|} (1-f_s(\delta)) e^{-2s} \, ds  \leq \frac{1}{\delta^4} + \bigO\big(\frac{|\omega_K|}{\delta^5}\big).
	$$
\end{corollary}

\begin{proof}
	Take $s_\pm$ defined in Lemma \ref{lem:begin_study_fsdelta}. We write $\mu(\,]\delta,+\infty[\,) = 2 \int_0^{+\infty}(1-f_s(\delta)) e^{-2s}\, ds$ and decompose the integral between $0$, $s_-$, $s_+$, $2\ln(\delta)$ and $+\infty$. Thanks to the first two points of Lemma \ref{lem:begin_study_fsdelta}, for $\delta \geq 4|\omega_K|$, we have
	\begin{equation}\label{eq:decompo_tail}
	\mu(\,]\delta,+\infty[\,) = 2 \int_{0}^{s_-} (1-f_s(\delta)) e^{-2s} \, ds + 2 \int_{s_+}^{2\ln(\delta)} (1-f_s(\delta)) e^{-2s} \, ds + \frac{1}{\delta^4}.
	\end{equation}
	The lower bound is then immediately obtained using the inequalities $f_s(\delta) \leq 1$ and $s_- \geq 2\ln|\omega_K|$.
	
	For the upper bound, we first notice the estimates, as $\delta \to +\infty$,
	$$
	s_- = 2\ln|\omega_K| + \bigO(\frac{|\omega_K|}{\delta}) \mbox{ and } s_+ = 2 \ln (\delta - \bigO(|\omega_K|)).
	$$
	We compute, as $\delta \to +\infty$,
	\begin{align*}
	2 \int_{s_+}^{2\ln(\delta)} (1-f_s(\delta)) e^{-2s} \, ds & \leq  2 \int_{s_+}^{2\ln(\delta)} e^{-2s} \, ds
	\\ & = - \frac{1}{\delta^4} + e^{-2s_+} = -\frac{1}{\delta^4}+\frac{1}{\delta^4}\big(1+\bigO\big(\frac{|\omega_K|}{\delta}\big)\big) = \bigO\big( \frac{|\omega_K|}{\delta^5}\big).
	\end{align*}
	Finally, the remaining integral on $[0,s_-]$ in Equation \eqref{eq:decompo_tail} can be restricted to $[0,2\ln|\omega_K|]$ for $\delta$ large enough so that $\ln(\frac{\delta}{2}) \geq s_- \geq 2\ln|\omega_K|$ (here we used the decreasing convergence $s_- \to 2\ln|\omega_K|$).
\end{proof}

In the cases $\OOO_K = \ZZ[i]$ or $\ZZ[j]$, we have $|\omega_K|=1$ so that Corollary \ref{cor:tail} provides the following tail estimate (illustrated in Figure \ref{fig:tail_approx_zi}), proving the last point of Theorem \ref{th:intro}: as $\delta\to+\infty$, we have
\begin{equation}\label{eq:tail_estimate_zij}
\mu(\, ]\delta, +\infty[\,) = \frac{1}{\delta^4} + \bigO \big(\frac{1}{\delta^5}\big).
\end{equation}

\begin{figure}[ht]
\centering
\scalebox{1}{
\begin{adjustbox}{clip, trim=1.7cm 0.5cm 1.7cm 1.4cm, max width=\textwidth}
	\includegraphics{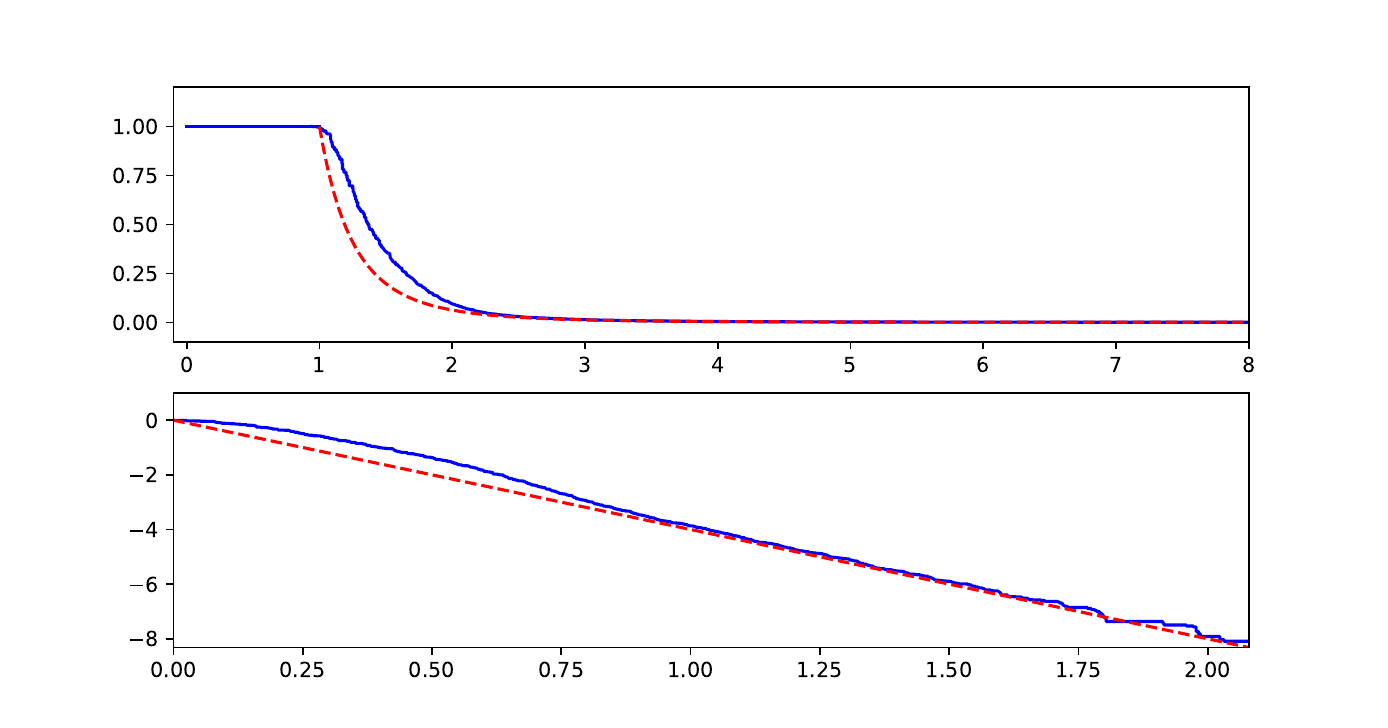}
\end{adjustbox}
}
\caption{At the top, the empirical tail distribution $\delta \to \mu_t(\, ]\delta, +\infty[\,)$ (in \textcolor{blue}{blue}) and the graph of $\delta \mapsto \frac{1}{\delta^4}$ (in \textcolor{red}{red}) using Gaussian fractions and the height value $e^\frac{t}{2}=30$. At the bottom, a logarithmic version of the top graph: we illustrate the estimate \eqref{eq:tail_estimate_zij} by comparing the functions $\ell \mapsto \ln(\mu_t(]e^\ell, +\infty[\,))$ and $\ell \mapsto -4\ell$.}
\label{fig:tail_approx_zi}
\end{figure}
In all other cases, we have $|\omega_K| > 1$, and an estimate of the remaining integral $2 \int_{0}^{2\ln|\omega_K|} (1-f_s(\delta)) e^{-2s} \, ds$ has yet to be computed. We conjecture this term not to be negligible with respect to $\frac{1}{\delta^4}$ when $|\omega_K| > 1$, because of numerical tests. An illustration is given in the case $K=\QQ(i\sqrt{163})$ (with principal ring of integers) in Figure \ref{fig:tail_approx_zsqrt163}.

\begin{figure}[!ht]
	\centering
	\scalebox{1}{
		\begin{adjustbox}{clip, trim=1.7cm 0.5cm 1.7cm 6.6cm, max width=\textwidth}
			\includegraphics{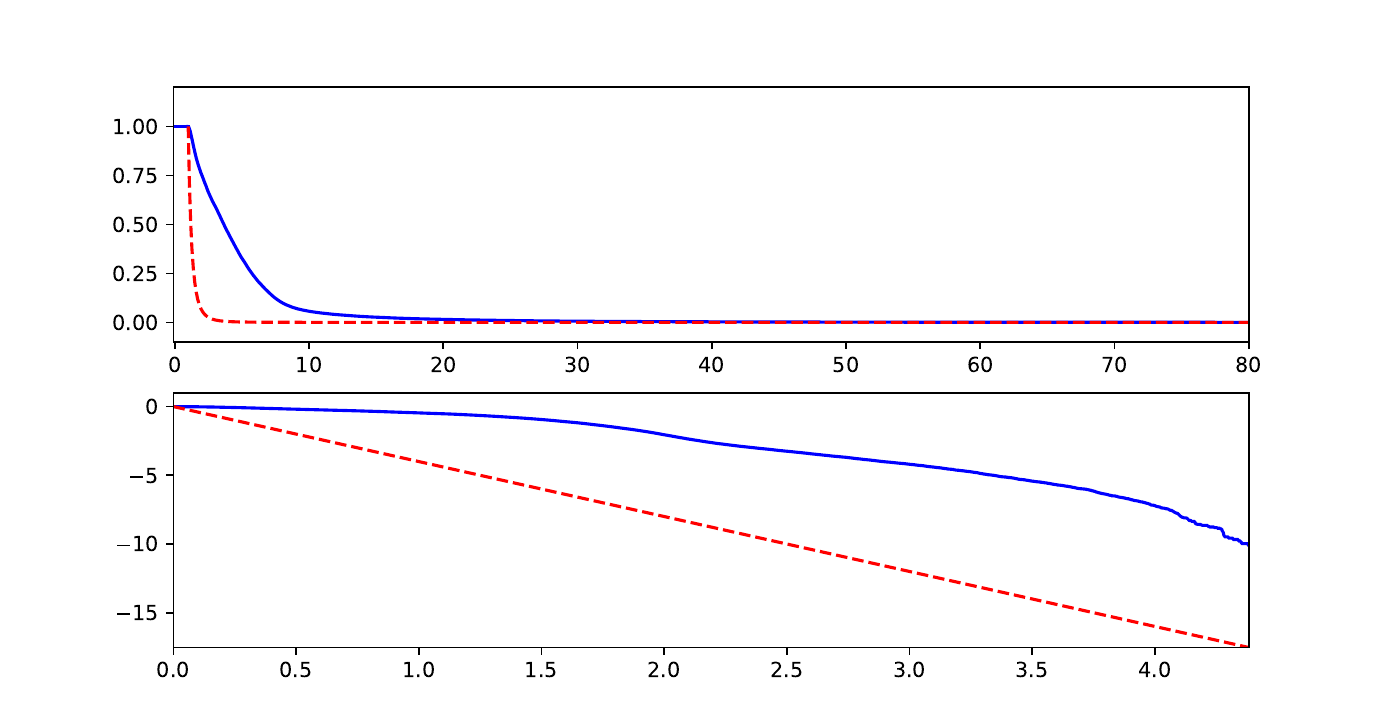}
		\end{adjustbox}
	}
	\caption{Similarly to the bottom graphs in Figure \ref{fig:tail_approx_zi}, a logarithmic comparison between the empirical tail distribution $\delta \to \mu_t(\, ]\delta, +\infty[\,)$ (in \textcolor{blue}{blue}) and the function $\delta \mapsto \frac{1}{\delta^4}$ (in \textcolor{red}{red}) for $\OOO_K=\ZZ[\frac{1+i\sqrt{163}}{2}]$. The estimate \eqref{eq:tail_estimate_zij} is conjectured not to hold in this case.}
	\label{fig:tail_approx_zsqrt163}
\end{figure}

\section{Further questions}
One could try and improve Theorem \ref{th:existence_density} by finding a closed formula for the asymptotic gap cumulative distribution function $\delta \mapsto \mu([0,\delta])$ (or equivalently, for the density $\frac{d \mu}{d \Leb_{[0,+\infty[}}$), which seems to be a really demanding challenge. A probably less involved improvement could be the computation of a tail estimate similar to Equation \eqref{eq:tail_estimate_zij} for other discriminants than $-4$ and $-3$.

In another direction, we could impose one of the two following conditions on the complex Farey fractions we study: a congruence condition on $q$ or changing the class of ideal generated by $p$ and $q$. See \cite{boca2014gapsfareydivisibilityconstraints} and \cite[§~7]{heersink2016poincaresetions} for a study of gaps with such conditions using the classical real Farey fractions, and \cite{boca2022paircorrfarey} for a study of pair correlations in the same setting. Fixing two nonzero ideals  $\mmm$ and $\mmm'$ of $\OOO_K$ with $\mmm \mid \mmm'$, one could study the asymptotic gap statistics, as $t \to +\infty$, for the set
$$
\F_{t,\mmm,\mmm'} = \big\{ \frac{p}{q} \mod \OOO_K \, : \, p\in \mmm, q\in\mmm', \, p \OOO_K + q\OOO_K = \mmm \mbox{ and } 0 < N(\mmm)^{-1} |q|^2 \leq e^t\big\}.
$$
We expect that there exists a probability measure, similar to the measure $\mu$ from Theorem \ref{th:existence_density} (and perhaps equal to $\mu$, depending on the choice of rescaling), which describes the asymptotic gap statistic.

When $\mmm=\OOO_K$, this study is pretty straightforward: we can use the same dynamic as in Section \ref{ssec:homogeneous dynamic}, except we replace $\Gamma$ by its Hecke subgroup
$$\Gamma_{\mmm'}=\Big\{\begin{pmatrix} a & b \\ c & d \end{pmatrix} \in \PSL_2(\OOO_K) \, : \, c \in \mmm'\Big\} = {\rm Stab}_{\PSL_2(\OOO_K)}\Big(\Big\{ \begin{pmatrix} p\\q \end{pmatrix} \in \widehat{\OOO_K^{\; 2}} \, : \, q \in \mmm' \Big\}\Big).$$
Then the proofs of \cite[Cor.~4.2]{parkkonenpaulin2023jointfarey} and Corollary \ref{cor:adapted_PaPa23} can be easily adapted to obtain a completely similar joint equidistribution result for $\F_{t,\OOO_K,\mmm'}$ on $(\CC/\OOO_K) \times M \bs G / \Gamma_{\mmm'}$, which can then be used to prove that the asymptotic gap distribution is the same for $\F_{t,\OOO_K,\mmm'}$ as for $\F_t$.

When $\mmm$ is not principal on the other hand, one would need some volume estimates around another cusp in $\partial_\infty \HH^3_{\RR}$, namely around the cusp associated to the fractional ideal class of $\mmm$, in order to adapt the result \cite[Cor.~4.2]{parkkonenpaulin2023jointfarey} to this study.

Finally, another generalisation could be obtained by the asymptotic study of gaps for Farey fractions in any number field $K$. Denoting by $\iota : K \to \RR^r \times \CC^c$ the canonical embedding, one could conduct a study of asymptotic gaps (i.e.~the nearest neighbour statistic for a well chosen distance on $\RR^r \times \CC^c \mod \iota(\OOO_K)$) for
\begin{align*}
\F_{T_1, \ldots, T_{r+c}}=\big\{ \big(\frac{p_1}{q_1}, \ldots, \frac{p_{r+c}}{q_{r+c}}\big) \mod \iota(\OOO_K) \, : & \, (p_1, \ldots, p_{r+c}), (q_1,\ldots, q_{r+c}) \in \iota(\OOO_K),
\\ & 0 < |q_1| \leq T_1, \, \ldots, 0< |q_{r+c}| \leq T_{r+c} \big\},
\end{align*}
perhaps adding a coprime condition $p\OOO_K + q\OOO_K=\OOO_K$ where $(p_1, \ldots, p_{r+c}) = \iota(p)$ and $(q_1,\ldots, q_{r+c})=\iota(q)$. This project is likely to begin with the study of the gaps in the Farey fractions of a real quadratic number field, and would require a higher rank analog of Theorem \ref{th:PaPa23} for Hilbert modular varieties.

\AtNextBibliography{\small}
\printbibliography[heading=bibintoc, title={References}]

\bigskip
{\small
	\noindent
	\begin{tabular}{l}
		Department of Mathematics and Statistics, P.O.~Box 35,\\
		40014 University of Jyv\"askyl\"a, FINLAND.\\
		{\it e-mail: sayousr@jyu.fi}
	\end{tabular}
}

\smallskip
and
\smallskip

{\small
	\noindent
	\begin{tabular}{l}
		Laboratoire de Mathématiques d'Orsay, UMR 8628 CNRS,\\
		Universit\'e Paris-Saclay, 91405 ORSAY Cedex, FRANCE.\\
		{\it e-mail: rafael.sayous@universite-paris-saclay.fr}
	\end{tabular}
}

\end{document}